\documentclass[oneside,english]{amsart}
\usepackage[T1]{fontenc}
\usepackage[latin9]{inputenc}
\usepackage{mathtools}
\usepackage{amstext}
\usepackage{amsthm}
\usepackage{amssymb}
\usepackage[all]{xy}

\makeatletter
\numberwithin{figure}{section}
\theoremstyle{plain}
\newtheorem{thm}{\protect\theoremname}
  \theoremstyle{definition}
  \newtheorem{defn}[thm]{\protect\definitionname}
  \theoremstyle{definition}
  \newtheorem{example}[thm]{\protect\examplename}
  \theoremstyle{plain}
  \newtheorem{lem}[thm]{\protect\lemmaname}
  \theoremstyle{remark}
  \newtheorem{rem}[thm]{\protect\remarkname}
  \theoremstyle{plain}
  \newtheorem{prop}[thm]{\protect\propositionname}

\usepackage{amsfonts}\usepackage{amsthm}\usepackage{enumerate}\usepackage{mathrsfs}\usepackage{cancel}\usepackage{url}
\usepackage{MnSymbol}\usepackage{hyperref}
\usepackage{bbm}

\input xy\huge
\xyoption{all}

\DeclareFontFamily{OT1}{pzc}{}
\DeclareFontShape{OT1}{pzc}{m}{it}{<-> s * [1.10] pzcmi7t}{}
\DeclareMathAlphabet{\mathpzc}{OT1}{pzc}{m}{it}

\newcommand{\rr}{\mathbb{R}}

\newcommand{\pp}{\mathbb{P}}

\newcommand{\ff}{\mathcal{F}}

\newcommand{\kk}{\mathcal{K}}

\newcommand{\cc}{\mathbb{C}}

\newcommand{\zz}{\mathbb{Z}}

\newcommand{\lc}{\mathcal{L}}
\newcommand{\xx}{\mathcal{X}}
\newcommand{\yy}{\mathcal{Y}}
\newcommand{\mm}{\mathcal{M}}

\newcommand{\im}{\operatorname{Im}}
\newcommand{\coker}{\operatorname{coker}}

\newcommand{\cl}{\mathcal{C}}

\newcommand{\manc}{\textbf{Man}^c}

\newcommand{\tb}{\mathbb{T}}

\newcommand{\zc}{\mathcal{Z}}

\newcommand{\basic}{\mathfrak{b}}

\newcommand{\aru}[2]{\ar[#1]^-{#2}}

\newcommand{\lf}{\mathsf{l}}
\newcommand{\kf}{\mathsf{k}}

\newcommand{\wc}{\mathcal{W}}
\usepackage{tikz} 
\usetikzlibrary{matrix,arrows}
\usepackage{tensor}
\newcommand{\fibp}[2]{\tensor[_{#1\thinspace}]{\times}{_{\thinspace #2}}}
\newcommand{\ev}{\operatorname{ev}}

\newcommand{\Sym}{\operatorname{Sym}}

\newcommand{\Or}{\operatorname{Or}}

\newcommand{\id}{\operatorname{id}}
\newcommand{\pt}{\operatorname{pt}}

\newcommand{\Aut}{\operatorname{Aut}}

\newcommand{\sstar}{\smallstar}

\makeatother

\usepackage{babel}
  \providecommand{\definitionname}{Definition}
  \providecommand{\examplename}{Example}
  \providecommand{\lemmaname}{Lemma}
  \providecommand{\propositionname}{Proposition}
  \providecommand{\remarkname}{Remark}
\providecommand{\theoremname}{Theorem}

\begin{document}

\author{Amitai Netser Zernik}\thanks{Hebrew University, \url{amitai.zernik@mail.huji.ac.il}}

\title{Moduli of Open Stable Maps to a Homogeneous Space}
\begin{abstract}
For $L\hookrightarrow X$ a Lagrangian embedding associated with a
real homogeneous variety, we construct the moduli space of stable
holomorphic discs mapping to $\left(X,L\right)$ as an orbifold with
corners equipped with a group action. Some essential constructions
involving orbifolds with corners are also discussed, including the
existence of fibered products and pushforward and pullback of differential
forms with values in a local system.
\end{abstract}

\maketitle
\tableofcontents{}

\section{\label{sec:Main results}Introduction}

Let $\left(X,\omega\right)$ be a closed symplectic manifold and
$L\subset X$ be a Lagrangian submanifold. We are interested in invariants
derived from stable maps of discs to $X$, whose boundary is required
to lie on $L$. The first step in producing such invariants is to
construct something akin to a singular chain from the moduli spaces
of such maps, and requires introducing perturbations that make the
Cauchy-Riemann operator regular. Proving the existence of the desired
perturbation data and keeping track of the choices involved is a
highly non-trivial task, which has been tackled using different approaches,
each with contributions by numerous authors (we will not attempt to
list them here).

In contrast, we will show that if the pair $\left(X,L\right)$ is
a real homogeneous pair, \emph{no perturbations are necessary}, and
the moduli space of stable disc-maps $\overline{\mm}_{0,k,l}\left(X,L,\beta\right)$
can be obtained from the moduli space of closed maps $\overline{\mm}_{0,n}\left(X,A\right)$
through a sequence of simple geometric constructions.

This is an open analog of a result, due to Fulton and Pandharipande
\cite{fulton-pandharipande}, that the moduli space of stable maps
to a convex non-singular projective variety, such as an algebraic
homogeneous space, is a smooth Deligne-Mumford stack whose associated
coarse moduli space is projective. The starting point for our construction
is the corresponding analytical statement, due to Robbin, Ruan and
Salamon \cite[Remark 3.13]{moduli-maps}: if a Kähler manifold $X$
is endowed with a transitive action of a compact Lie group, then the
moduli space $\overline{\mm}_{0,n}\left(X,A\right)$ is a compact,
complex orbifold (without boundary).

We turn to a precise statement of our main result. 
\begin{defn}
\label{def:real homogeneous variety}A \emph{real homogeneous variety
}is a tuple
\[
\left(X,\omega,J,G_{X},\alpha,c_{G},c_{X}\right)
\]
where: 
\begin{itemize}
\item $X$ is a compact Kähler manifold, with symplectic form $\omega$
and an $\omega$-compatible integrable complex structure $J$, 
\item $G_{X}$ is a compact lie group, 
\item $\alpha:G_{X}\times X\to X$ defines a transitive action of $G_{X}$
on $X$ which preserves $\omega$ and $J$, and
\item $c_{G}:G_{X}\to G_{X}$ and $c_{X}:X\to X$ are a pair of involutions
such that $c_{G}$ is a group homomorphism, $c_{X}$ is anti-symplectic
and anti-holomorphic, and ${c_{X}\,\alpha\left(g,x\right)=\alpha\left(c_{G}\,g,c_{X}\,x\right)}$.
\end{itemize}
A \emph{real homogeneous pair} $\left(X,L=X^{\zz/2}\right)$ is
a pair where $X$ is a real homogeneous variety and $L=X^{\zz/2}$
denotes the submanifold of real, or $c_{X}$-fixed, points of $L$. 
\end{defn}
If $\left(X,L\right)$ is a real homogeneous pair, 
\[
i_{L}:L\hookrightarrow X
\]
is a Lagrangian embedding and the action of the $c_{G}$-invariant
subgroup $G=G_{X}^{\zz/2}$ on $X$ preserves $L$. It is not hard
to check that the induced map $\mathfrak{g}:=T_{0}G\to T_{x}L$ is
surjective for every $x\in L$. 

If $k,l$ are non-negative integers, we write
\[
G_{k,l}=G\times\Sym\left(k\right)\times\Sym\left(l\right)
\]
for the product of $G$ with the symmetric groups on $k$ and on $l$
elements. Section \ref{sec:pf moduli as orbifold} is devoted to proving
the following theorem.
\begin{thm}
\label{thm:moduli as orbifold}Let $\left(X,L\right)$ be a real homogeneous
pair. Let $k,l$ be non-negative integers and $\beta\in H_{2}\left(X,L\right)$.
Suppose $\beta\neq0$ or $k+2l\geq3$. Then the moduli space 
\[
\overline{\mm}_{0,k,l}\left(\beta\right)=\overline{\mm}_{0,k,l}\left(X,L,\beta\right)
\]
parameterizing families of stable holomorphic disc-maps of degree
$\beta$ with $k$ boundary marked points and $l$ interior marked
points, is a compact $G_{k,l}$\emph{-orbifold with corners}, admitting
a $G_{k,l}$-equivariant map\emph{ }
\[
\overline{\mm}_{0,k,l}\left(\beta\right)\xrightarrow{f}\overline{\mm}_{0,k+2l}\left(\beta+\overline{\beta}\right).
\]
There's a unique $G_{k,l}$-equivariant map 
\[
\overline{\mm}_{0,k,l}\left(\beta\right)\xrightarrow{\ev}L^{k}\times X^{l}
\]
such that 
\[
\ev_{c}\circ f=\left(i_{L}^{k}\times\left(\left(\id_{X}^{l}\times c_{X}^{l}\right)\circ\Delta_{X^{l}}\right)\right)\circ\ev
\]
. 
\end{thm}
Here 
\[
\overline{\mm}_{0,k+2l}\left(\beta+\overline{\beta}\right)=\overline{\mm}_{0,k+2l}\left(X,\beta+\overline{\beta}\right)
\]
denotes the moduli space of stable genus zero maps of class $\beta+\overline{\beta}$
(see (\ref{eq:double homology})) with $k+2l$ marked points, and
$\ev_{c}:\overline{\mm}_{0,k+2l}\left(\beta+\overline{\beta}\right)\to X^{k+2l}$
is the associated evaluation map. The precise meaning of ``a compact
$G_{k,l}$-orbifold with corners'' and ``a $G_{k,l}$-equivariant
map'' is discussed in Section \ref{sec:appendix}.
\begin{example}
For every positive integer $n$, $\left(X,L\right)=\left(\cc\pp^{n},\rr\pp^{n}\right)$
is a real homogeneous pair with $G_{X}=U\left(n+1\right)$ the group
of $\left(n+1\right)\times\left(n+1\right)$ unitary matrices, acting
by the restriction of the standard $GL_{n+1}\left(\cc\right)$ group
action on $\cc\pp^{n}$, and with $c_{G}$ and $c_{X}$ given by
\[
\left[c_{G}\left(A\right)\right]_{i,j}:=\overline{A_{i,j}}
\]
and 
\[
c_{X}\left(\left[z_{0}:\cdots:z_{n}\right]\right):=\left[\overline{z_{0}}:\cdots:\overline{z_{n}}\right].
\]
By Theorem \ref{thm:moduli as orbifold} the moduli spaces $\overline{\mm}_{0,k,l}\left(\cc\pp^{n},\rr\pp^{n},\beta\right)$
are compact orbifolds with corners.
\end{example}
Doing away with perturbations has several technical advantages. For
instance, one can given an elementary construction of the Fukaya $A_{\infty}$
algebra of $L\subset X$ using pullback and pushforward of differential
forms, see Solomon and Tukachinsky \cite{jake+sara-A8}. Moreover,
the $G$-equivariant extension of various constructions is considerably
simpler than if perturbations have to be taken into account, see \cite{equiv-OGW-invts}. 

The close relationship between the moduli spaces of discs and the
moduli space of curves is also quite useful - see for example the
computation of the torus fixed points and their tubular neighborhoods
in \cite{fp-loc-OGW}.

We can summarize the construction of $\overline{\mm}_{0,k,l}\left(\beta\right)$
by the following $G$-equivariant diagram 

\begin{equation}
\overline{\mm}_{0,k,l}\left(\beta\right)\overset{s}{\hookrightarrow}\mm^{::}\xrightarrow{o}\widetilde{\mm}\xrightarrow{B}\overline{\mm}_{0,k+2l}\left(\beta+\overline{\beta}\right)^{\zz/2}\xrightarrow{i}\overline{\mm}_{0,k+2l}\left(\beta+\overline{\beta}\right).\label{eq:open-closed relation}
\end{equation}
The moduli space $\overline{\mm}_{0,k+2l}\left(\beta+\overline{\beta}\right)$
of closed maps is a complex orbifold of the expected dimension, since
the associated Cauchy-Riemann operator is regular at every point.
We denote by $\overline{\mm}_{0,k+2l}\left(\beta+\overline{\beta}\right)^{\zz/2}$
the stacky fixed-points of this space, with respect to an anti-holomorphic
involution which conjugates the map and swaps some of the markings.
One should think of $\overline{\mm}_{0,k+2l}\left(\beta+\overline{\beta}\right)^{\zz/2}$
as parametrizing the \emph{double} of the disc-map; points of $\overline{\mm}_{0,k+2l}\left(\beta+\overline{\beta}\right)^{\zz/2}$
are represented by a stable map $\left(\Sigma\xrightarrow{u}X,...\right)$
together with an anti-holomorphic involution of the domain $b:\Sigma\to\Sigma$,
that swaps some of the marked points and fixes others (see Lemma \ref{lem:Z/2 fps}).
Following Liu \cite[\S 2.3]{LiuModuli} we call the data $\left(\left(\Sigma\xrightarrow{u}X,...\right),b\right)$
a \emph{symmetric configuration}. To halve the double and recover
the disc map we're interested in, we must choose a fundamental domain
$\Sigma^{1/2}\subset\Sigma$ for the involution $b$, whose boundary\footnote{More precisely, we have $\partial\Sigma^{1/2}=\Sigma^{b}$ only at
configurations which do \emph{not} have an E-type node. Such a node
represents the degeneration of $\partial\Sigma$ to a single point,
see more about this below.} is the $b$-fixed points of $\Sigma$:
\[
\Sigma^{1/2}\cap b\left(\Sigma^{1/2}\right)=\Sigma^{b},
\]
see Definition \ref{def:fundamental domains}. Indeed, the space $\mm^{::}$
will parameterize all possible \emph{fundamental} \emph{configurations}
$\left(\left(\Sigma\xrightarrow{u}X,...\right),b,\Sigma^{1/2}\right)$,
and the map $\mm^{::}\xrightarrow{B\circ o}\overline{\mm}_{0,k+2l}\left(\beta+\overline{\beta}\right)^{\zz/2}$
will be the forgetful map 
\[
\left(\left(\Sigma\xrightarrow{u}X,...\right),b,\Sigma^{1/2}\right)\mapsto\left(\left(\Sigma\xrightarrow{u}X,...\right),b\right).
\]
To construct $\mm^{::}$, we must first cut up $\overline{\mm}_{0,k+2l}\left(\beta+\overline{\beta}\right)^{\zz/2}$
along the locus of configurations with \emph{real}, or $b$-invariant,\emph{
}nodes. We call this process a \emph{hyperplane blowup }(see $\S$\ref{subsec:Hyperplane-Blowup}).
It produces the map 
\[
\widetilde{\mm}\xrightarrow{B}\overline{\mm}_{0,k+2l}\left(\beta+\overline{\beta}\right)^{\zz/2}.
\]
The choice of fundamental domain then corresponds to choosing an element
of the map 
\[
\mm^{::}\xrightarrow{o}\widetilde{\mm},
\]
which is a 2-sheeted cover of its image. 

There are two ways to see why we need to introduce the hyperplane
blowup $B$. First, note that sometimes $\Sigma^{b}=\emptyset$, as
in the following example adapted from Liu, \cite[Example 3.6]{LiuModuli}.
For $-1\leq\epsilon\leq1$ the maps
\begin{align*}
\Sigma=\left\{ \left[x:y:z\right]\in\cc\pp^{2}|x^{2}+y^{2}+\epsilon z^{2}=0\right\}  & \to X=\cc\pp^{1}\\
\left[x:y:z\right] & \mapsto\left[x:y\right]
\end{align*}
with the standard conjugation action on the domain and range, define
a path
\[
\gamma:\left[-1,1\right]\to\overline{\mm}_{0,0,0}\left(\cc\pp^{1},2\right)^{\zz/2}.
\]
For $\epsilon>0$ we have $\Sigma^{b}=\emptyset$, which are configurations
we'd like to discard since they're not the double of any disc map.
For $\epsilon<0$ we have a valid configuration, obtained as the
double of a degree two map $\left(D^{2},\partial D^{2}\right)\to\left(\cc\pp^{1},\rr\pp^{1}\right)$.
At $\epsilon=0$ the boundary of the disc map degenerates to a real
node of type E (see Definition \ref{def:types of real nodes}, and
\cite[Definition 3.4]{LiuModuli}). We find that $\gamma|_{\left[-1,0\right]}$
admits a (non-unique) lift to a path in $\overline{\mm}_{0,0,0}\left(\cc\pp^{1},\rr\pp^{1},\left(1,1\right)\right)$,
where the boundary of the disc-map shrinks to $\left[0:0:1\right]\in L$.
In other words, the first reason to blow up is so we can discard unwanted
symmetric topologies (see \cite[\S 2.3]{LiuModuli}) - such configurations
cannot be in the image of the map $\mm^{::}\xrightarrow{B\circ o}\overline{\mm}_{0,k+2l}\left(\beta+\overline{\beta}\right)^{\zz/2}$.

We turn to the second reason to introduce the hyperplane blowup $B$.
 Consider some configurations with some number $r\geq1$ of \emph{type
H} nodes, so $\Sigma^{b}/\nu$ consists of $r+1$ circles glued at
$r$ pairs of points, and the number of possible fundamental domains
(i.e., the size of the fiber of the map $B\circ o$) is $2^{r+1}$.
This also suggests we should cut up the moduli space. Indeed, the
map $B$ is locally modeled on the gluing of orthants: $2^{r}\times\rr^{n-r}\times[0,\infty)^{r}\to\rr^{n}$.
Picking one of the $2^{r}$ inverse images of $0\in\rr^{n}$ amounts
to picking a smoothing direction for the $r$ nodes, and there are
precisely two fundamental domains compatible with every such choice
of smoothing directions. 

With $\widetilde{\mm}$ in place, the map $\mm^{::}\xrightarrow{o}\widetilde{\mm}$
is the composition of a 2-sheeted cover (forgetting the choice of
fundamental domain $\Sigma^{1/2}\subset\Sigma$) with the inclusion
of a clopen component (omitting configurations with $\Sigma^{b}=\emptyset$).

Finally, we restrict to a clopen component $\overline{\mm}_{0,k,l}\left(\beta\right)\overset{s}{\hookrightarrow}\mm^{::}$,
corresponding to those fundamental domains $\Sigma^{1/2}$ such that
(i) $u_{*}\left(\left[\Sigma^{1/2},\Sigma^{b}\right]\right)=\beta$
and (ii) $\Sigma^{1/2}$ contains a specific subset of the $k+2l$
marked points. Put another way, once we restrict to the clopen component
\[
\mm_{\beta}^{::}\hookrightarrow\mm^{::}
\]
of those fundamental domains such that (i) holds, we need to take
the quotient 
\[
\mm_{\beta}^{::}\to\overline{\mm}_{0,k,l}\left(\beta\right)
\]
by the free $\left(\zz/2\right)^{\times l}$ group action where the
$i$'th $\zz/2$ factor acts by swapping the labels of the $k+i$
and $k+l+i$ markings (these two markings should be indistinguishable,
representing the same interior marked point). Instead, we take a section
of this quotient
\[
\overline{\mm}_{0,k,l}\left(\beta\right)\hookrightarrow\mm_{\beta}^{::}
\]
by restricting further the clopen component to include only those
configurations where the fundamental domain contains a particular
representative of each pair $\left\{ k+i,k+l+i\right\} $.

Section \ref{sec:appendix} contains essential constructions and results
related to orbifolds with corners. This section is written with a
view towards subsequent applications, and thus covers significantly
more than is strictly necessary for the proof of Theorem \ref{thm:moduli as orbifold}.
Most notably, we discuss
\begin{itemize}
\item the existence of fibered products in the category of orbifolds with
corners, relying on the work of Joyce \cite{joyce-generalized} on
manifolds with corners.
\item Differential forms, local systems, pushforward and pullback operations.
\item The notion of a hyperplane blowup.
\item Group actions on orbifolds with corners.
\end{itemize}
\textbf{Acknowledgments.} I am deeply grateful to my teacher, Jake
Solomon. For one, it was his suggestion to construct the moduli space
of stable disc maps starting from the moduli space of closed maps.
I was partially supported by ERC starting grant 337560, ISF Grant
1747/13 and NSF grant DMS-1128155.

\section{\label{sec:pf moduli as orbifold}Proof of Theorem \ref{thm:moduli as orbifold} }

\subsection{The groupoid of disc-map configurations}

We begin by reviewing the notion of a stable disc-map and isomorphisms
of such maps. In Liu \cite[\S 5.1]{LiuModuli}, the general notion
of a stable map from a Riemann surfaces with boundary to an arbitrary
pair $\left(X,L\right)$, is given. In general, the moduli spaces
of such maps admit only a Kuranishi structure with corners, see \cite[Theorem 1.2]{LiuModuli}.
In contrast, we will see that the moduli spaces of disc maps to a
real homogeneous pair can be given the structure of an orbifold with
corners. 

Let $\left(X,L\right)$ be a real homogeneous pair as in Definition
\ref{def:real homogeneous variety}. A tuple $\left(k,l,\beta\right)$,
where $k$ and $l$ are non-negative integers and $\beta\in H_{2}\left(X,L\right)$,
is called a \emph{moduli specification }if either $\beta\neq0$ or
$k+2l\geq3$. Fix some moduli specification $\basic=\left(k,l,\beta\right)$.
Consider the groupoid $D_{\basic}$ of \emph{$\left(k,l,\beta\right)$-disc-map
configurations}, whose objects consist of tuples $\left(\Sigma,\kappa,\lambda,\nu,u\right)$,
where
\begin{itemize}
\item $\Sigma$ is a possibly disconnected compact Riemann surface with
boundary and $\nu:\left(\Sigma,\partial\Sigma\right)\to\left(\Sigma,\partial\Sigma\right)$
is an involution with finitely many orbits of size 2. The configuration
is connected and has genus zero. That is to say, each connected component
of $C\subset\Sigma$ is diffeomorphic to $D^{2}$ or to $\cc\pp^{1}$,
and the graph $\left(\pi_{0}\left(\Sigma\right),\nu\right)$ is in
fact a tree.
\end{itemize}
Here 
\begin{equation}
\left(\pi_{0}\left(\Sigma\right),\nu\right)\label{eq:incidence graph}
\end{equation}
is the \emph{incidence graph}, whose vertices are the connected components
of $\Sigma$, and where each orbit of size two $o$ of $\nu$ determines
an edge, incident to the connected components $o$ intersects. We
denote by $\mathring{\Sigma}^{\nu}$ and $\partial\Sigma^{\nu}$ the
fixed points of $\nu$ in the interior and boundary of $\Sigma$,
respectively. These define the \emph{smooth points }of the orbit space
$\Sigma/\nu$. The orbits of size two are called \emph{nodes}.

\begin{itemize}
\item $\kappa:\left[k\right]\hookrightarrow\partial\Sigma^{\nu}$ and $\lambda:\left[l\right]\hookrightarrow\mathring{\Sigma}^{\nu}$
are injective maps.
\end{itemize}
Note that we do \emph{not} assume the points $\im\kappa$ appear in
any particular order around $\partial\Sigma$.
\begin{itemize}
\item $u:\left(\Sigma,\partial\Sigma\right)\to\left(X,L\right)$ is a $\nu$-invariant
holomorphic map with ${u_{*}\left[\Sigma,\partial\Sigma\right]=\beta\in H_{2}\left(X,L\right)}$.
\item For each connected component $C\subset\Sigma$, we have $u_{*}\left[C,\partial C\right]\neq0$
or 
\[
\left(\left|\kappa^{-1}\left(\partial C\right)\right|+\left|\partial C\backslash\Sigma^{\nu}\right|\right)+2\left(\left|\lambda^{-1}\left(C\right)\right|+\left|\mathring{C}\backslash\Sigma^{\nu}\right|\right)\geq3.
\]
\end{itemize}
There's an arrow in $\mathcal{D}_{\basic}$ connecting two objects
$\left(\Sigma,\kappa,\lambda,\nu,u\right)$ and $\left(\Sigma',\kappa',\lambda',\nu',u'\right)$
for every biholomorphism $\phi:\Sigma\to\Sigma'$ preserving all of
the additional structure. The product $G_{\basic}=G\times\Sym\left(k\right)\times\Sym\left(l\right)$
of the compact lie group $G$ with the permutation groups acts on
$\mathcal{D}_{\basic}$ by translating maps and relabeling markings.

We will show that the groupoid $\mathcal{D}_{\basic}$ is equivalent
to a $G_{\basic}$-orbifold with corners $\overline{\mm}_{0,k,l}\left(\beta\right)$.
The construction proceeds from right to left, along the diagram (\ref{eq:open-closed relation}).

\subsection{Complex moduli of closed maps}

Let $n$ be a non-negative integer and $A\in H_{2}\left(X\right)$.
Following \cite{moduli-maps} we consider $\left(n,A\right)$-\emph{rational
configurations.} An $\left(n,A\right)$-\emph{rational configuration}
is a tuple $\left(\Sigma,\nu,\lambda,u\right)$ where $\Sigma$ is
a disjoint union of $\cc\pp^{1}$'s, $\nu$ is an involution with
finitely many orbits of size two, and the incidence graph $\left(\pi_{0}\left(\Sigma\right),\nu\right)$
is required to be a tree. $\lambda$ is an injective map $\left[n\right]\to\Sigma^{\nu}$,
and $u:\Sigma\to X$ is a $\nu$-invariant holomorphic map with $u_{*}\left[\Sigma\right]=A\in H_{2}\left(X\right)$.
Finally, for every connected component $C\subset\Sigma$ if $u_{*}\left[C\right]=0$
then $\left|\left(C\backslash\Sigma^{\nu}\right)\coprod\lambda^{-1}\left(C\right)\right|\geq3$.
The collection of configurations are the objects of a groupoid, in
which there's an arrow $\left(\Sigma,\nu,\lambda,u\right)\to\left(\Sigma',\nu',\lambda',u'\right)$
for every biholomorphism $\phi:\Sigma\to\Sigma'$ respecting all of
the structure. 

By \cite[Remark 3.13]{moduli-maps} if $X$ is a real homogeneous
variety then for any $A\in H_{2}\left(X\right)$ the moduli space
of stable holomorphic maps $\overline{\mm}_{0,n}\left(A\right)$ is
a compact complex orbifold. 

The compact group $G_{X}^{+}=G_{X}\times\Sym\left(n\right)$ acts
on $\overline{\mm}_{0,n}\left(A\right)$ by translating maps and relabeling
markings. This action can be constructed by direct analogy with the
construction of the $\zz/2$-action given in $\S$\ref{subsec:Z/2 action}
below.

In (\ref{eq:open-closed relation}), we've set $n=k+2l$, and taken
$A=\beta+\overline{\beta}\in H_{2}\left(X\right)$, which is defined
as follows. There's a map 
\begin{equation}
H_{2}\left(X,L\right)\ni\beta\mapsto\beta+\overline{\beta}\in H_{2}\left(X\right),\label{eq:double homology}
\end{equation}
which takes a singular chain $\sigma\in C_{2}\left(X\right)$ with
$\partial\sigma\in C_{1}\left(L\right)\subset C_{1}\left(X\right)$
representing $\beta\in H_{2}\left(X,L\right)$ to the homology class
represented by the cycle $\sigma+\left(c_{X}\right)_{*}\sigma$ (recall
$c_{X}:X\to X$ is an antiholomorphic involution which fixes $L$).

\subsection{\label{subsec:Orbifolds-as-Stacks}Orbifolds as Stacks}

We briefly review how one can think about orbifolds in terms of stacks,
following Metzler \cite{metzler} and Pronk \cite{pronk}. See also
Lerman \cite{lerman} for a very pleasant explanation of this approach.
A \emph{smooth stack} $\mathcal{C}$ is a category together with a
functor $\mathcal{C}\xrightarrow{\pi}\text{\textbf{Man}}$ to the
category of smooth manifolds, such that (i) $\mathcal{C}$ is a category
fibered in groupoids over $\text{\textbf{Man}}$ and (ii) a certain
descent condition is satisfied. Given an orbifold without boundary
$\mm$ (which we take to mean an object of the bicategory of fractions
$\text{\textbf{Orb}}_{\partial\neq\emptyset}$ given by a groupoid
in $\text{\textbf{Man}}$; see Remark \ref{rem:complex orbifolds}),
we obtain a stack by taking $\mathcal{C}$ to be the category whose
objects are 1-cells $T\to\mm$ where $T$ is a manifold (without boundary)
and morphisms between $T_{1}\xrightarrow{f_{1}}\mm$ to $T_{2}\xrightarrow{f_{2}}\mm$
are smooth maps $T_{1}\xrightarrow{g}T_{2}$ together with a 2-cell
$f_{1}\Rightarrow f_{2}\circ g$. The functor $\pi$ sends $T\to\mm$
to $T$.

The category of smooth maps $\mm\xrightarrow{}\mm$ is equivalent
to the full subcategory of the category of functors $\mathcal{C}\to\mathcal{C}$,
consisting of those functors that commute with $\pi$. In particular,
a 2-cell $\alpha:c^{2}\Rightarrow\id$ is given simply by a natural
transformation between functors $\cl\to\cl$. 

This perspective is especially useful for thinking about $\overline{\mm}_{0,n}\left(A\right)$.
The category of holomorphic maps $T\to\overline{\mm}_{0,n}\left(A\right)$
is equivalent to the category of \emph{stable families of maps }over
$T$, i.e. stable families of maps of type $\left(0,n,A\right)$ in
the sense of \cite[3.1]{moduli-maps}, which are of the form 
\[
\left(\pi:Q\to T,S_{1},...,S_{n}:T\to Q,H:Q\to X\right).
\]
We recall that this means: (i) $\pi,S_{*},H$ are holomorphic maps,
(ii) $\pi$ is only allowed certain nodal singularities, (iii) for
each $t\in T$ the fiber $\pi^{-1}\left(t\right)$ is a nodal Riemann
surface of genus zero, and $H|_{\pi^{-1}\left(t\right)}$ represents
$A$, (iv) $S_{1},...,S_{n}$ are sections of $\pi$, and (v) $\left(\pi^{-1}\left(t\right),S_{1}\left(t\right),...,S_{n}\left(t\right),H|_{\pi^{-1}\left(t\right)}\right)$
is a stable map. An arrow 
\[
\left(\pi:Q\to T,S_{*}:T\to Q,H:Q\to X\right)\to\left(\pi':Q'\to T,S_{*}':T\to Q',H':Q'\to X\right)
\]
is a holomorphic morphism $Q\to Q'$ which preserves the additional
data (cf. \cite[\S 3.1]{moduli-maps}).

This allows us to define the 1-cells $\overline{\mm}_{0,n}\left(A\right)\to\overline{\mm}_{0,n}\left(A\right)$
and the 2-cells between them simply by describing functorial manipulations
of stable families of maps. We can extend the discussion to include
anti-holomorphic maps (an anti-holomorphic map $T\to\overline{\mm}_{0,n}\left(A\right)$
is just a holomorphic map $\overline{T}\to\overline{\mm}_{0,n}\left(A\right)$,
arrows between such maps are still given by holomorphic morphisms
of families $Q\to Q'$ as above), so we can also consider $\overline{\mm}_{0,n}\left(A\right)$
as a stack over the category of complex manifolds with both holomorphic
and antiholomorphic maps.

\subsection{\label{subsec:Z/2 action}A $\zz/2$-action on $\mm=\overline{\mm}_{0,n}\left(\beta+\overline{\beta}\right)$}

Consider $\mm=\overline{\mm}_{0,n}\left(\beta+\overline{\beta}\right)$
as a complex stack. There's an antiholomorphic map 
\begin{equation}
c_{\mm}:\overline{\mm}_{0,n}\left(\beta+\overline{\beta}\right)\to\overline{\mm}_{0,n}\left(\beta+\overline{\beta}\right)\label{eq:moduli-conj}
\end{equation}
which sends a point $p$ represented by the $\left(n,\beta+\overline{\beta}\right)$-rational
configuration $\left(\Sigma,\nu,\lambda,u\right)$ to a point represented
by 
\[
\overline{\left(\Sigma,\nu,\lambda,u\right)}:=\left(\overline{\Sigma},c_{\Sigma}\circ\nu\circ c_{\Sigma}^{-1},c_{\Sigma}\circ\lambda\circ c_{n}^{-1},c_{X}\circ u\circ c_{\Sigma}^{-1}\right),
\]
with $c_{\Sigma}:\Sigma\to\overline{\Sigma}$ the map which replaces
the almost complex structure $j$ on $\Sigma$ by $-j$, and $c_{n}:\left[n\right]\to\left[n\right]$
the involution 
\begin{equation}
\left(\begin{array}{ccccccccccc}
1 & \cdots & k &  & k+1 & \cdots & k+l &  & k+l+1 & \cdots & k+2l\\
1 & \cdots & k &  & k+l+1 & \cdots & k+2l &  & k+1 & \cdots & k+l
\end{array}\right).\label{eq:markings involution}
\end{equation}
To construct (\ref{eq:moduli-conj}) as a map of stacks, start from
any family of maps 
\[
\left(\pi_{B}:Q\to B,S_{1},...,S_{n}:B\to Q,H_{B}:Q\to X\right)
\]
and map it to the conjugate family of maps 
\begin{equation}
\left(\overline{Q}\xrightarrow{\overline{\pi}_{B}}\overline{B},\left(\overline{S}_{c_{n}\left(i\right)}:\overline{B}\to\overline{Q}\right)_{i=1}^{n},c_{X}\circ\overline{H}_{B}\right).\label{eq:conjugate family}
\end{equation}
This is functorial and thus defines the desired map of stacks $c=c_{\mm}:\mm\to\mm$.
There's a 2-cell 
\begin{equation}
\alpha:c^{2}\Rightarrow\mbox{id}\label{eq:2-cell}
\end{equation}
which is given by the obvious arrow $\overline{\overline{Q}}\to Q$,
considered as a natural transformation between functors-of-families.
It satisfies a coherence condition, which is an equality of two
2-cells $c^{3}\Rightarrow c$. Thus, (\ref{eq:moduli-conj}) and (\ref{eq:2-cell})
define a $\zz/2$-action on $\mm$.

In fact, we will need a bit more than this, namely that the $G_{X}^{+}$
action extends to a $G_{X}^{++}=G_{X}^{+}\rtimes\zz/2$ action on
$\mm$, where the semidirect product is defined using the homomorphism
$\zz/2\to\Aut\left(G_{X}^{+}\right)$ in which the generator acts
by 
\begin{equation}
G_{X}^{+}=G_{X}\times\Sym\left(n\right)\ni\left(\phi,\sigma\right)\mapsto\left(c_{G}\left(\phi\right),c_{n}\circ\sigma\circ c_{n}^{-1}\right).\label{eq:Z/2 action on G_C}
\end{equation}
It is straightforward to construct the data of this group action in
terms of functors and natural transformations of families, as above.
We will focus our attention on the $\zz/2$ action to keep the notation
palatable.

\subsection{The $\zz/2$-fixed points of $\mm$}

Let $\mm^{\zz/2}=\overline{\mm}_{0,n}\left(\beta+\overline{\beta}\right)^{\zz/2}$\textbf{
}denote the $\zz/2<G_{X}^{++}$-stacky fixed points of $\mm$, and
let $G_{+}$ denote the group of elements of $G_{X}^{+}$ that are
fixed under (\ref{eq:Z/2 action on G_C}): 
\begin{equation}
G_{+}=\left(G_{X}^{+}\right)^{\zz/2}=G\times\Sym\left(k\right)\times\Sym\left(l\right)\times\left(\zz/2\right)^{l}.\label{eq:G_+ def}
\end{equation}

\begin{lem}
\label{lem:Z/2 fps}(a) $\mm^{\zz/2}$ is a $G_{+}$-orbifold (without
boundary) with 
\[
\dim_{\rr}\mm^{\zz/2}=\frac{1}{2}\dim_{\rr}\mm
\]
and the map $\mm^{\zz/2}\xrightarrow{i}\mm$ is a proper closed immersion\footnote{See Definition \ref{def:properties of orbi-maps}. We warn the reader
that this is \emph{not }the usual meaning in algebraic geometry; rather
it is a generalization of the notion of closed immersion of manifolds
without boundary in differential geometry. In particular, $i$ will
not be injective in general.}. In particular, $\mm^{\zz/2}$ is compact.

(b) As a groupoid $\mm^{\zz/2}$ is equivalent to the groupoid whose
points are represented by pairs $\left(\left(\Sigma,\nu,\lambda,u\right),b\right)$,
where $\left(\Sigma,\nu,\lambda,u\right)$ is a rational $\left(n,\beta+\overline{\beta}\right)$-configuration
and $b:\left(\Sigma,\nu,\lambda,u\right)\to\overline{\left(\Sigma,\nu,\lambda,u\right)}$
is an arrow in $\mm$ such that $\overline{b}\circ b$ is the identity
map, considered as a biholomorphism $\left(\Sigma,\nu,\lambda,u\right)\to\overline{\overline{\left(\Sigma,\nu,\lambda,u\right)}}$.

\end{lem}
\begin{proof}
We prove part (a). Note that $\mm$ is in fact represented by a proper
étale groupoid $M_{1}\rightrightarrows M_{0}$ where $M_{i}$ are
complex manifolds and the structure maps $s,t,e,i,m$ are local biholomorphisms.
This implies we can consider $\mm$ as a \emph{holomorphic stack},
or stack over the category of complex manifolds with holomorphic maps
(cf. \cite[Definition 28]{metzler}). As a map between holomorphic
stacks, $M=M_{0}\xrightarrow{a}\mm$ is a \emph{holomorphic atlas}.
That is, a map whose domain is equivalent to a complex manifold, and
such that the weak fibered product $W=T\times_{\mm}M_{0}$ with any
other map of holomorphic stacks from a complex manifold $T\to\mm$
is a complex manifold, and the pullback map $W\to T$ is a surjective
local biholomorphism.

We will now use a trick to construct a manifold $P'$ with $\dim P'=\frac{1}{2}\dim\mm$
and an étale covering $P'\xrightarrow{a'}\mm^{\zz/2}$ for $\mm^{\zz/2}$,
in the sense of \cite[Proposition 75]{metzler}. Since the diagonal
$\mm^{\zz/2}\to\mm^{\zz/2}\times\mm^{\zz/2}$ is proper, this will
imply that the conditions of \cite[Proposition 75]{metzler} are met
and $\mm^{\zz/2}$ is a smooth orbifold as claimed.

To construct $P'$, let $\overline{M}$ be the complex conjugate
of $M$, so that $c\circ\overline{a}:\overline{M}\to\mm$ is a also
a holomorphic atlas for $\mm$. Let $P$ be the complex manifold representing
the 2-fiber product 
\[
P=M\fibp{a}{c\circ\overline{a}}\overline{M}.
\]
The map $P\to M\to\mm$ is also a holomorphic atlas. A holomorphic
map $S\to P$, or \emph{$S$-point}, is given by 
\begin{equation}
\left(S\xrightarrow{x}M,S\xrightarrow{y}\overline{M},a\circ x\xRightarrow{b}c\circ\overline{a}\circ y\right)\label{eq:S-point}
\end{equation}
Here $\xRightarrow{b}$ is a morphism of the category $\mm$, between
the objects $a\circ x$ and $c\circ\overline{a}\circ y$. We define
a strict antiholomorphic involution on $P$ by specifying its action
on the $S$-points for all $S$. Namely the involution sends (\ref{eq:S-point})
to
\[
\left(S\xrightarrow{y}M,S\xrightarrow{x}\overline{M},a\circ y\xRightarrow{c\left(b\right)^{-1}\alpha^{-1}}c\circ\overline{a}\circ x\right).
\]
This is an involution since
\[
c\left(\alpha_{y}\right)c^{2}\left(b\right)\alpha_{x}^{-1}=\alpha_{c\left(y\right)}c^{2}\left(b\right)\alpha_{x}^{-1}=b\alpha_{x}\alpha_{x}^{-1}=b.
\]
The first equality is the coherence condition, and the second is naturality
of the transformation $\alpha$. Now let $P'$ denote the fixed-points
of the involution. Using the slice theorem and linear algebra, the
fixed-points of an antiholomorphic involution form a real submanifold
$P'\subset P$ with $\dim P'=\frac{1}{2}\dim P$. It is easy to see
that the $S$-points of $P'$ are pairs, consisting of a map $S\xrightarrow{x}M$
together with an arrow $a\circ x\xrightarrow{b}c\circ a\circ x$,
such that 
\begin{equation}
b=c\left(b\right)^{-1}\alpha^{-1}\label{eq:stacky fp}
\end{equation}
(we do not care about complex structures from this point inwards,
so we write $a$ and not $\overline{a}$, etc.). Eq (\ref{eq:stacky fp})
is the defining relation for a $\zz/2$-fixed point, cf. \cite[Proposition 2.5]{stacky-action}.
In other words, we have a map of smooth stacks $P'\xrightarrow{a'}\mm^{\zz/2}$.
To prove this is an étale covering, it suffices to show that $a'$
is the pullback of the étale covering $a$. More precisely, we claim
that the square
\[
\xymatrix{P'\ar[d]_{a'}\ar[r]^{i'} & M\ar[d]^{a}\\
\mm^{\zz/2}\ar[r]_{i} & \mm
}
\]
is 2-cartesian, where $i'$ is the composition $P'\to P=M\times_{\mm}\overline{M}\to M$.
Indeed, an $S$-point of $\mm^{\zz/2}\fibp{i}{a}M$ is represented
by $\left(\pi,r,g\right)$ where $S\xrightarrow{\pi}\mm^{\zz/2}$
is an $S$-fixed-point, $S\xrightarrow{g}M$ is an $S$-point of $M$,
and $a\,g\xRightarrow{r}i\,\pi$ is an arrow of $\mm$. Let $f:=i\pi$,
so $\pi$ is represented by some arrow $f\xrightarrow{b}c\circ f$
in $\mm$ satisfying (\ref{eq:stacky fp}). Now it is easy to check
that 
\[
a\,g\xrightarrow{c\left(r\right)^{-1}\circ b\circ r}c\left(a\,g\right)
\]
satisfies (\ref{eq:stacky fp}). Since this works for any $S$, we've
constructed an arrow 
\[
\mm^{\zz/2}\fibp{i}{a}M\to P'.
\]
It is now straightforward to construct the reverse arrow $P'\to\mm^{\zz/2}\fibp{i}{a}M$
and show these form an equivalence of stacks.

The map $P'\to M$, hence the map $i$, is seen to be a closed immersion.
Since the fiber of $i$ is finite, it is proper, and so $\mm^{\zz/2}$
is compact. Since $\mm$ was equipped with a $G_{X}^{++}$ action,
$\mm^{\zz/2}$ is equipped with a residual $G_{+}$-action (see
\cite[Remark 2.4]{stacky-action}).

Part (b) is straightforward.
\end{proof}
\begin{defn}
\label{def:fp configuration}We call $\left(\left(\Sigma,\nu,\lambda,u\right),b\right)$
as in part (b) of Lemma \ref{lem:Z/2 fps} a \emph{symmetric configuration. }
\end{defn}
Note a fixed configuration $\left(\left(\Sigma,\nu,\lambda,u\right),b\right)$
defines a ``symmetric Riemann surface with $\left(l,k\right)$ marked
points'' in the sense of \cite[\S 2.2.2]{LiuModuli}.

\subsection{\label{subsec:blowup M tilde}The blowup $\widetilde{\mm}$ of $\mm^{\zz/2}$.}

The next step is to blow up a hyper subset in $\mm^{\zz/2}$ (the
motivation for this was discussed in the introduction). Let $\left(\left(\Sigma,\nu,\lambda,u\right),b\right)$
be a $\zz/2$-fixed configuration\emph{. }Let $o=\left\{ o_{1},o_{2}\right\} $
be a node of the configuration, that is, an orbit of $\nu$ of size
2. We say $o$ is a \emph{real }node if $b\left(o\right)=o$. 
\begin{defn}
\label{def:types of real nodes}A real node $o=\left\{ o_{1},o_{2}\right\} $,
$b\left(o\right)=o$, is of one of two types (cf. \cite[Definition 3.4]{LiuModuli}):
\begin{enumerate}
\item If $b\left(o_{1}\right)=o_{1}$ we say $o$ is a \emph{type H node}. 
\item If $b\left(o_{1}\right)=o_{2}$ we say $o$ is a \emph{type E node}.
 
\end{enumerate}
\end{defn}
Nodes of type H correspond to strip breaking in Floer theory. Nodes
of type E correspond to disc configurations where the boundary has
degenerated to a point. Since we're considering only genus zero configurations,
if $o$ is a type E node it is the only real node.

In $\S$\ref{subsec:orbifold hyper bu} we construct the hyperplane
blowup of an orbifold $\xx=X_{1}\overset{s,t}{\rightrightarrows}X_{0}$
along a hyper subset $E\subset X_{0}$, $s^{-1}E=t^{-1}E$. In our
case, we take $E$ to be the subset of all points represented by configurations
with a real node. 

We now explain why this is a hyper subset. First, we reformulate some
well-known properties of the divisor of nodal configurations\footnote{In algebraic geometry this is often called the ``boundary'' divisor.
This terminology becomes especially confusing in our context, so we
avoid it.}.

Let $B_{n,\beta+\overline{\beta}}$ denote set of all 4-tuples $\left(\lf_{1},A_{1},\lf_{2},A_{2}\right)$
with $\lf_{i}\subset\left[n\right]$ and $A_{i}\in H_{2}\left(X\right)$,
such that: (i) $A_{1}+A_{2}=\beta+\overline{\beta}$, (ii) $\lf_{1}\coprod\lf_{2}=\left[n\right]$,
and (iii) $A_{i}=0\Rightarrow\left|\lf_{i}\right|\geq3$ for $i=1,2$.
We denote 
\begin{equation}
H_{\basic}=\left\{ \left(\lf_{1},A_{1},\lf_{2},A_{2}\right)\in B_{n,\beta+\overline{\beta}}|c_{n}\left(\lf_{i}\right)=\lf_{i}\text{ and }c_{X*}A_{i}=A_{i}\text{ for }i=1,2\right\} \label{eq:H_basic def}
\end{equation}
\begin{equation}
E_{\basic}=\left\{ \left(\lf_{1},A_{1},\lf_{2},A_{2}\right)\in B_{n,\beta+\overline{\beta}}|c_{n}\left(\lf_{i}\right)=\lf_{2-i}\text{ and }c_{X*}A_{i}=A_{2-i}\text{ for }i=1,2\right\} .\label{eq:E_basic def}
\end{equation}

\begin{rem}
\label{rem:E_b =00003D emptyset}Note that $E_{\basic}=\emptyset$
unless $k=0$ and $\beta+\overline{\beta}\in\im\left(\id+c_{X*}\right)$. 
\end{rem}
We have maps between orbifolds without boundary
\[
D_{H_{\basic}}^{\cc}:\left(\coprod_{H_{\basic}}\overline{\mm}_{0,\lf_{1}\coprod\left\{ \sstar_{1}\right\} }\left(A_{1}\right)\times_{X}\overline{\mm}_{0,\lf_{2}\coprod\left\{ \sstar_{2}\right\} }\left(A_{2}\right)\right)_{\zz/2}\to\overline{\mm}_{0,n}\left(\beta+\overline{\beta}\right)
\]
and 
\[
D_{E_{\basic}}^{\cc}:\left(\coprod_{E_{\basic}}\overline{\mm}_{0,\lf_{1}\coprod\left\{ \sstar_{1}\right\} }\left(A_{1}\right)\times_{X}\overline{\mm}_{0,\lf_{2}\coprod\left\{ \sstar_{2}\right\} }\left(A_{2}\right)\right)_{\zz/2}\to\overline{\mm}_{0,n}\left(\beta+\overline{\beta}\right)
\]
defined by gluing. Henceforth, a subscript such as $\lf_{1}\coprod\left\{ \sstar_{1}\right\} $
means we label the markings by the indicated finite set instead of
$\left\{ 1,...,n\right\} $. The fibered products are over the evaluation
maps at $\sstar_{1},\sstar_{2}$, which are transverse (cf. \cite{mcduff+salamon}).
The $\zz/2$ subscript denotes the stacky quotient by the action which
swaps the two factors. The map 
\[
D_{k,l,\beta}^{\cc}=D_{H_{\basic}}^{\cc}\coprod D_{E_{\basic}}^{\cc}
\]
is a faithful, proper, closed immersion with transversal self-intersection
(cf. Definition \ref{def:orbimaps properies 2}). Indeed, the essential
fiber $F=\left(D_{k,l,\beta}^{\cc}\right)^{-1}\left(p\right)$ over
$p=\left[\Sigma,\kappa,\lambda,\nu,u\right]$ is in natural bijection
with the subset of nodes $\left\{ \left\{ a_{i},b_{i}\right\} \right\} _{i=1}^{r}$
of $\nu$ that partition the markings $\left[n\right]$ and the degree
$\beta+\overline{\beta}$ into 4-tuples as in (\ref{eq:H_basic def})
or in (\ref{eq:E_basic def}). For any\emph{ }subset of nodes the
associated map of conormal bundles 
\[
\bigoplus_{i=1}^{r}\left(\mathbb{L}_{a_{i}}^{\vee}\otimes\mathbb{L}_{b_{i}}^{\vee}\right)^{\vee}\to T^{\vee}\overline{\mm}_{0,n}\left(\beta+\overline{\beta}\right)
\]
is injective. Here $\mathbb{L}_{a_{i}}^{\vee}$ denotes the tangent
line to the universal curve at $a_{i}$.

We make $D_{H_{\basic}}^{\cc}$,$D_{E_{\basic}}^{\cc}$, and thus
also $D_{k,l,\beta}^{\cc}$, into $\zz/2$-equivariant maps. $\zz/2$
acts on the domain of $D_{H_{\basic}}^{\cc}$ by the product of conjugation
maps, where we treat $\sstar_{1}$ and $\sstar_{2}$ as fixed markings
(i.e. we identify them with one of the first $k$ markings in (\ref{eq:markings involution})).
The $\zz/2$ action on the domain of $D_{E_{\basic}}^{\cc}$ swaps
the two moduli factors, sending $\sstar_{1}$ to $\sstar_{2}$ and
vice-versa. Passing to $\zz/2$ invariants we obtain a hyper map
(cf. Definition \ref{def:orbimaps properies 2})

\begin{equation}
\mathcal{W}_{\kf,\lf,\beta}=\left(\coprod\overline{\mm}_{0,\lf_{1}\coprod\sstar_{1}}\left(A_{1}\right)^{\zz/2}\times_{L}\mm_{0,\lf_{2}\coprod\sstar_{2}}^{\zz/2}\left(A_{2}\right)\right)_{\zz/2}\xrightarrow{D_{\kf,\lf,\beta}}\overline{\mm}_{0,n}\left(\beta+\overline{\beta}\right)^{\zz/2},\label{eq:real divisor}
\end{equation}
with $\im D_{\kf,\lf,\beta}=E$. We let $B:\widetilde{\mm}\to\mm=\overline{\mm}_{0,n}\left(\beta+\overline{\beta}\right)^{\zz/2}$
denote the hyperplane blowup of $\mm^{\zz/2}$ along $E$. It is a
compact $G_{+}$-orbifold. $\widetilde{\mm}$ is equivalent to the
groupoid of symmetric configurations together with a choice of one
of two possible smoothing directions for each real node (this choice
is equivalent to the choice of ``incident orthant'', see the proof
of Proposition \ref{prop:hyper bu functor ManC})

\subsection{The forgetful map. }

Let $\basic=\left(k,l,\beta\right)$ be a basic moduli specification,
let $\basic_{+}=\left(k+1,l,\beta\right)$, and write $n=k+2l$. The
forgetful map 
\[
\pi^{\cc}:\overline{\mm}_{0,n+1}\left(\beta+\overline{\beta}\right)\to\overline{\mm}_{0,n}\left(\beta+\overline{\beta}\right)
\]
induces a map of fixed points $\pi:\overline{\mm}_{0,n+1}\left(\beta+\overline{\beta}\right)^{\zz/2}\to\overline{\mm}_{0,n}\left(\beta+\overline{\beta}\right)^{\zz/2}$.
Let $\mm_{\text{no }E}^{\zz/2}\subset\overline{\mm}_{0,n}\left(\beta+\overline{\beta}\right)^{\zz/2}$
denote the open suborbifold consisting of symmetric configurations
with no nodes of type E, let $\mm_{\text{no }E,+}^{\zz/2}=\pi^{-1}\left(\mm_{\text{no }E}^{\zz/2}\right)$,
and let $\widetilde{\mm}_{\text{no }E},\widetilde{\mm}_{\text{no }E,+}$denote
the corresponding hyperplane blowups at the loci of configurations
with real nodes.

\begin{lem}
$\pi$ lifts to a b-fibration (see Definition \ref{def:properties of orbi-maps})
\begin{equation}
\tilde{\pi}:\widetilde{\mm}_{\text{no }E,+}\to\widetilde{\mm}_{\text{no }E}.\label{eq:forgetful map}
\end{equation}
\end{lem}
\begin{proof}
The claim can be checked locally on the domain. Let $\sigma_{+}=\left(\left(\Sigma_{+},\nu_{+},\lambda_{+},u_{+}\right),b_{+}\right)$
and $\sigma=\left(\left(\Sigma,\nu,\lambda,u\right),b\right)$ be
symmetric configurations so that $\pi$ maps
\[
p=\left[\sigma_{+}\right]\in\overline{\mm}_{0,n+1}\left(\beta+\overline{\beta}\right)^{\zz/2}
\]
to
\[
q=\left[\sigma\right]\in\overline{\mm}_{0,n}\left(\beta+\overline{\beta}\right)^{\zz/2}.
\]
This means $\sigma$ is obtained from $\sigma_{+}$ by erasing the
marked point $\lambda_{+}\left(k+1\right)$ and contracting the incident
component $C\subset\Sigma_{+}$ if it becomes unstable. Let $x_{1},...,x_{N+1}$
be local coordinates for $\overline{\mm}_{0,n+1}\left(\beta+\overline{\beta}\right)^{\zz/2}$
centered around $p$ and $y_{1},...,y_{N}$ be local coordinates for
$\overline{\mm}_{0,n}\left(\beta+\overline{\beta}\right)^{\zz/2}$
centered around $q$. There are three cases to consider.
\begin{enumerate}
\item No component is contracted. In this case we can choose the coordinates
so that the germ $\pi_{p}$ is given by 
\[
\left(y_{1},...,y_{N}\right)=\left(x_{2},...,x_{N}\right),
\]
with $x_{2},...,x_{r+1}$ and $y_{1},...,y_{r}$ the smoothing parameters
for the real nodes of $p$ and of $q$, respectively. Clearly this
map lifts to the hyperplane blowup and is a b-fibration.
\item The component $C$ is contracted to a node of type H in $\sigma$.
In this case, we can choose the coordinates so $\pi_{p}$ is given
by 
\begin{equation}
\left(y_{1},...,y_{N}\right)=\left(x_{1}\cdot x_{2},x_{3},...,x_{N}\right),\label{eq:pi_p case (2)}
\end{equation}
with $x_{1},x_{2}$ the smoothing parameters of the two nodes of $\sigma_{+}$
incident to $C$, $y_{1}$ the image H-node, and $x_{3},...,x_{r+1}$
are smoothing parameters for the other nodes of $\sigma_{+}$, with
$y_{2},...,y_{r}$ smoothing parameters for the corresponding nodes
of $\sigma$. Again, this map lifts to the hyperplane blowup. It is
a b-fibration (cf. \cite[Example 4.4(ii)]{joyce-generalized})
\item The component $C$ is contracted to a node of type E in $\sigma$.
In this case, we can choose the coordinates so $\pi_{p}$ is given
by 
\[
\left(y_{1},...,y_{N}\right)=\left(x_{1}^{2}+x_{2}^{2},x_{3},...,x_{N}\right).
\]
Here $x_{1}+ix_{2},x_{1}-ix_{2}$ are the smoothing parameters for
the complex-conjugate nodes incident to $C$. $y_{1}$ is a smoothing
parameter for the E-node. $\sigma,\sigma_{+}$ have no other real
nodes in this case. The map $\rr^{2}\to[0,\infty)$, $\left(x_{1},x_{2}\right)\mapsto\left(x_{1}^{2}+x_{2}^{2}\right)$
is \emph{not }smooth (as a map of manifolds with corners it is only
\emph{weakly smooth}; cf. \cite[Example 2.3(i)]{joyce-generalized}
for a closely related example). This is why we discard such configurations
from the codomain.
\end{enumerate}
\end{proof}

\begin{rem}
In studying the open Gromov-Witten theory of $\left(\cc\pp^{2m},\rr\pp^{2m}\right)$
E-nodes are excluded since the moduli spaces $\overline{\mm}_{0,k,l}\left(\beta\right)$
there either have $k>0$ or else $\beta\in H_{2}\left(X,L\right)=\zz$
is odd, which implies $\beta+\overline{\beta}\not\in\im\left(1+c_{X*}\right)$.
Either way there are no E-type nodes by Remark \ref{rem:E_b =00003D emptyset}.
In other applications, one may be content to consider the forgetful
map as a weakly smooth map only, though this requires a revision of
the category of orbifolds with corners as defined in Section \ref{sec:appendix},
which we will not pursue here. Finally, an $S^{1}$-blowup may be
used to resolve the problem: the map $[0,\infty)\times S^{1}\to[0,\infty)$
\[
\left(r,\theta\right)\mapsto\left(r\cos\theta,r\sin\theta\right)\mapsto r^{2}
\]
is a (smooth) b-fibration.
\end{rem}

\subsection{Picking a fundamental domain and the map $\mm^{::}\xrightarrow{o}\widetilde{\mm}$.}

If $\xx$ is an orbifold with corners, we denote by $\xx^{\circ}$
the interior, or depth zero, points of $\xx$. By construction, $\widetilde{\mm}^{\circ}$
is an open suborbifold of $\mm^{\zz/2}$, and its points are represented
by symmetric configurations $\sigma=\left(\left(\Sigma,...\right),b\right)$
with no additional data. 

Let $\widetilde{\mm}_{\neq\emptyset}\subset\widetilde{\mm}$ denote
the clopen component which is the closure of those points $\widetilde{\mm}^{\circ}$
 represented by a symmetric configuration $\left(\left(\Sigma,\nu,\lambda,u\right),b\right)$
with $\Sigma^{b}\neq\emptyset$. In fact, since there are no real
nodes, we conclude that for $q\in\widetilde{\mm}_{\neq\emptyset}^{\circ}\subset\mm^{\zz/2}$
we must have 
\[
\tilde{\pi}^{-1}\left(q\right)=\Sigma^{b}\simeq S^{1},
\]
which is orientable. Let $\lc$ denote the extension to the boundary
(cf. Lemma \ref{lem:ls extensions} (a)) of the fiber orientation
local system 
\begin{equation}
\left[\tilde{\pi}_{*}\left(\Or\left(T\widetilde{\mm}_{+}\right)\right)\otimes\Or\left(T\widetilde{\mm}\right)^{\vee}\right]|_{\widetilde{\mm}_{\neq\emptyset}^{\circ}}.\label{eq:fiber or ls}
\end{equation}
The unit sections of $\lc$ form a 2-sheeted cover $\mm^{::}\xrightarrow{}\widetilde{\mm}_{\neq\emptyset}$,
and we define $o$ to be the composition
\[
\mm^{::}\to\widetilde{\mm}_{\neq\emptyset}\hookrightarrow\widetilde{\mm}.
\]
We now describe $\mm^{::}$ and $B\circ o$ in terms of fundamental
domains.
\begin{defn}
\label{def:fundamental domains}(a) Let $\sigma=\left(\left(\Sigma,\nu,\lambda,u\right),b\right)$
be a symmetric configuration. A \emph{fundamental domain }$\Sigma^{1/2}$
for $\sigma$ is a subset $\Sigma^{1/2}\subset\Sigma$ biholomorphic
to a disjoint union of $\cc\pp^{1}$'s and $D^{2}$'s such that:

(i) $\Sigma=\Sigma^{1/2}\cup b\left(\Sigma^{1/2}\right)$

(ii) $\Sigma^{1/2}\cap b\left(\Sigma^{1/2}\right)=\Sigma^{b}$

(iii) $\nu|_{\Sigma\backslash\Sigma^{b}}\left(\Sigma^{1/2}\right)\subset\Sigma^{1/2}$
(in other words, $\Sigma^{1/2}$ either contains both sides of each
non-real node, or none of them).

(b) a \emph{fundamental configuration} consists of $\left(\left(\Sigma,\nu,\lambda,u\right),b,\Sigma^{1/2}\right)$
where $\sigma=\left(\left(\Sigma,\nu,\lambda,u\right),b\right)$ is
a symmetric configuration and $\Sigma^{1/2}$ is a fundamental domain
for $\sigma$.

(c) For $i=1,2$ let $\phi_{i}=\left(\left(\Sigma_{i},\nu_{i},\lambda_{i},u_{i}\right),b_{i},\Sigma_{i}^{1/2}\right)$
be fundamental configurations. A \emph{morphism }of fundamental configurations
$\phi_{1}\to\phi_{2}$ is a biholomorphism $\Sigma_{1}\to\Sigma_{2}$
that respects all of the additional data.
\end{defn}
\begin{lem}
$\mm^{::}$ is equivalent to the groupoid of\emph{ }fundamental configurations.
The map $B\circ o$ corresponds to the map
\[
\left(\left(\Sigma,\nu,\lambda,u\right),b,\Sigma^{1/2}\right)\mapsto\left(\left(\Sigma,\nu,\lambda,u\right),b\right).
\]
\end{lem}
\begin{proof}
Consider first an interior point $\widetilde{\mm}_{\neq\emptyset}^{\circ}$
represented by a symmetric configuration $\sigma=\left(\left(\Sigma,\nu,\lambda,u\right),b\right)$.
It is clear that there are two fundamental domains for $\sigma$ in
this case, and they induce opposite orientations on $\partial\Sigma=\Sigma^{b}\simeq S^{1}$
by the outward normal orientation convention. Clearly, $B\circ o$
is the map that forgets the chosen orientation and corresponding fundamental
domain.

Heuristically, we want to extend this correspondence continuously
to the boundary. To achieve this, we consider how the fundamental
domain changes in continuous families. Recall we constructed a holomorphic
atlas 
\[
P\to\mm=\overline{\mm}_{0,n}\left(\beta+\overline{\beta}\right),\;P=M\fibp{c}{c\circ a}\overline{M},
\]
equipped with an antiholomorphic involution $P\to P$ whose fixed
points form an atlas $P^{\zz/2}\to\mm^{\zz/2}$. The pullback $\widetilde{P}\to\widetilde{\mm}_{\neq\emptyset}$
along $\widetilde{\mm}_{\neq\emptyset}\to\mm^{\zz/2}$ is an atlas
for $\widetilde{\mm}_{\neq\emptyset}$ ($\widetilde{P}$ can also
be obtained by taking a clopen component of a hyperplane blowup of
$P^{\zz/2}$).

Consider a stable family of maps $\left(Q\to P,S_{\star},H\right)$
associated with $P\to\mm$ (cf. $\S$\ref{subsec:Orbifolds-as-Stacks}).
Let $\widetilde{Q}\to\widetilde{P}$ denote the topological pullback
along $\widetilde{P}\to P$. There's an involution $\widetilde{Q}\to\widetilde{Q}$
over $\id_{\widetilde{P}}$. The fiber over each $p\in\widetilde{P}$
is a nodal Riemann surface with an antiholomorphic involution, whose
normalization is a symmetric configuration representing the image
of $p$ in $\mm^{\zz/2}$. We abuse notation and treat the local system
$\lc$ as defined over $\widetilde{P}$, by pulling it back. By definition
(cf. Lemma \ref{lem:ls extensions}) there's a sufficiently small
open neighborhood $p\in U\subset\widetilde{P}$ such that the stalk
$\lc_{p}$ is in bijection with relative orientations for the circle
bundle $\widetilde{Q}^{\widetilde{b}}|_{U^{\circ}}\to U^{\circ}$.
Fix a germ $g\in\lc_{p}$. As we discussed in the first paragraph,
the corresponding relative orientation determines a \emph{family of
nodal fundamental domains} $Q^{1/2}\to U^{\circ}$. That is, $Q^{1/2}$
is the closure of connected component of $\widetilde{Q}|_{U^{\circ}}\backslash\widetilde{Q}^{\widetilde{b}}|_{U^{\circ}}$
such that the normalization $\hat{Q}_{q}^{1/2}$ of the fiber $Q_{q}^{1/2}$
over every point $q\in U^{\circ}$ is the fundamental domain associated
with the specified orientation of $\widetilde{Q}^{\widetilde{b}}|_{q}$.
We define the fundamental domain at $p\in\widetilde{P}$ corresponding
to $g$ to be the normalization of
\[
\overline{Q^{1/2}}\cap Q_{p},
\]
where the closure is taken in $\widetilde{Q}$. We need to check that
this is indeed a fundamental domain for a symmetric configuration
$\sigma=\left(\left(\Sigma,\nu,\lambda,u\right),b\right)$ representing
$B\left(p\right)$. Let us assume $\sigma$ has a single node which
is a real node of type H, the other cases are similar. In this case
the blowup over $B\left(p\right)$ is modeled on the gluing of half-planes
$\left\{ +,-\right\} \times[0,\infty)\times\rr^{N-1}\to\rr^{N}$.
The choice of $p\in\left\{ +,-\right\} \times\left\{ 0\right\} $
determines the direction the parameter $y_{1}$ approaches zero, where
$y_{1}=x_{1}\cdot x_{2}$ is the real smoothing parameter as in (\ref{eq:pi_p case (2)}).
It is easy to see that a consistent orientation for the hyperbolas
for values of $y_{1}$ approaching zero, determines a pair of orientations
for the branches $x_{1}=0$ and $x_{2}=0$ over $y_{1}=0$, and that
the family of fundamental domains converges to a pair of connected
components of $\Sigma\backslash\Sigma^{b}$ near the real node that
specify a fundamental domain for $\sigma$ inducing the given orientations
on the boundary
\end{proof}

\subsection{Admissible disc-configurations and $\overline{\mm}_{0,k,l}\left(\beta\right)\protect\overset{s}{\protect\hookrightarrow}\mm^{::}$}

Note that up until now, the construction only ``knew'' about $\beta+\overline{\beta}\in H_{2}\left(X\right)$.
In general we may have $\beta+\overline{\beta}=\beta'+\overline{\beta'}$
with $\beta\neq\beta'$. For an example of this, take ${\beta=\left(1,0\right)\in H_{2}\left(\cc\pp^{1},\rr\pp^{1}\right)}$
and $\beta'=\overline{\beta}=\left(0,1\right)$. 

Thus, we first restrict to a clopen component 
\[
\mm_{\beta}^{::}\subset\mm^{::}
\]
of points represented by fundamental configurations $\left(\left(\Sigma,\nu,\lambda,u\right),b,\Sigma^{1/2}\right)$
with 
\[
u_{*}\left[\Sigma^{1/2},\partial\Sigma^{1/2}\right]=\beta.
\]
Next, for each $1\leq i\leq l$, we want to identify the markings
$k+i$ and $k+l+i$; in other words, we have
\[
\overline{\mm}_{0,k,l}\left(\beta\right)=\left(\mm_{\beta}^{::}\right)_{\left(\zz/2\right)^{\times l}},
\]
the stacky quotient by the $\left(\zz/2\right)^{\lf}\triangleleft G_{+}$
action on $\mm_{\beta}^{::}$. It is naturally a $G_{k,l}=G_{+}/\left(\zz/2\right)^{l}$
orbifold with corners (see \cite[Remark 2.4]{stacky-action}), and
is clearly equivalent to the groupoid of $\left(k,l,\beta\right)$-disc
configurations.

In fact, it is not hard to see that the map $q$ admits a smooth
$G_{k,l}$-equivariant section $\mm_{\basic}\overset{s}{\hookrightarrow}\mm_{\basic}^{::}$
which we take to be the inclusion of the clopen component corresponding
to fundamental configurations where 
\[
\lambda\left(\left\{ k+1,...,k+2l\right\} \right)\cap\Sigma^{1/2}=\left\{ k+1,...,k+l\right\} .
\]

It is easy to check that there's a unique $G_{k,l}$-equivariant map
$\ev$ such that 
\[
\ev_{c}\circ f=\left(i_{L}^{k}\times\left(\id_{X}^{l}\times c_{X}^{l}\right)\circ\Delta_{X^{l}}\right)\circ\ev
\]
This completes the proof of Theorem \ref{thm:moduli as orbifold}.

\section{\label{sec:appendix}Orbifolds with Corners}

In this section we fix our notion of orbifold with corners, and introduce
some related constructions. Some care is required, since fiber products
in the category of manifolds with corners are somewhat elusive.
We will define the category of orbifolds with corners as the Pronk
2-localization \cite{pronk} of the bicategory of proper étale groupoids
in a category of manifolds with corners, obtained by formally inverting
étale equivalences. Our setup of the category of manifolds with corners
follows Joyce's work \cite{joyce-generalized} closely.

\subsection{Manifolds with corners}

We refer the reader to \cite[\S 2]{joyce-generalized} for the terminology
we use regarding manifolds with corners. The manifolds we'll consider
have ``ordinary'' corners (as opposed to generalized corners), which
are modeled on $\rr_{k}^{n}:=[0,\infty)^{k}\times\rr^{n-k}$. 

A \emph{weakly smooth} map $f:U\to V$ between open subsets $U\subset\rr_{k}^{m}$
and $V\subset\rr_{l}^{n}$ is a continuous map $f=\left(f_{1},...,f_{n}\right)$
such that all the partial derivatives \linebreak{}
${\frac{\partial^{a_{1}+\cdots+a_{m}}}{\partial u_{1}^{a_{1}}\cdots\partial u_{m}^{a_{m}}}f_{j}:U\to\rr}$
exist and are continuous (including one-sided derivatives where applicable). 

An $n$-dimensional \emph{manifold with corners} $X$ is a second
countable Hausdorff space equipped with a maximal $n$-dimensional
atlas of charts $\left(U,\phi\right)$ where $U\subset X$ is open
and $\phi:U\to\rr_{k}^{n}$ is a homeomorphism ($n$ is fixed, $k$
may vary), with weakly smooth transitions. A \emph{weakly smooth map
$f:X\to Y$} between manifolds with corners is a continuous map which
is of this form in every coordinate patch. A weakly smooth map $f:X\to Y$
is said to be \emph{smooth}, \emph{strongly smooth, interior, b-normal,
simple}, or a \emph{b-fibrations} as in \cite[Definitions 2.1, 4.3]{joyce-generalized}.
``A map'' between manifolds with corners will always be assumed
to be smooth unless specifically stated otherwise, and we denote by
$\manc$ the category of manifolds with corners with smooth maps.

The \emph{depth }of a point ${x=\left(x_{1},...,x_{n}\right)\in\rr_{k}^{n}}$
is defined by ${\mbox{depth}\left(x\right)=\#\left\{ 1\leq i\leq k|x_{i}=0\right\} }$.
It is easy to see that the transitions preserve the depth, so we can
speak of the depth of a point $x\in X$. We define $S^{k}\left(X\right)=\left\{ x\in X|\mbox{depth}\left(x\right)=k\right\} $.
A \emph{local k-corner component $\gamma$ }of $X$ at $x$ is a local
choice of connected component of $S^{k}\left(X\right)$ near $x$
(cf. \cite[Definition 2.7]{joyce-generalized}); a local 1-corner
component is also called a \emph{local boundary component}. 

We have manifolds with corners
\[
\partial X=C_{1}\left(X\right)=\left\{ \left(x,\beta\right)|x\in X,\,\mbox{\ensuremath{\beta}\ is a local boundary component of \ensuremath{X}at \ensuremath{x}}\right\} 
\]
and, for every $k\geq0$,
\[
C_{k}\left(X\right)=\left\{ \left(x,\gamma\right)|x\in X,\,\gamma\text{ is a local \ensuremath{k}-corner component of \ensuremath{X} at \ensuremath{x}}\right\} .
\]
Letting $\partial^{k}X$ denote the iterated boundary, we note that
$C_{k}\left(X\right)\simeq\partial^{k}X/\Sym\left(k\right)$ where
$\Sym\left(k\right)$ acts by permuting the local boundary components. 

We can consider $C\left(X\right)=\coprod_{k\geq0}C_{k}\left(X\right)$
as a \emph{local manifold with corners }(or ``manifold with corners
of mixed dimension'', in Joyce's terms). These form a category and
the various properties of maps can be used to describe maps between
local manifolds with corners. If $f:X\to Y$ is a smooth map of manifolds
with corners, there's an induced interior map 
\[
C\left(f\right):C\left(X\right)\to C\left(Y\right)
\]
We denote by $i_{X}^{\partial}:\partial X\to X$ the map defined by
$i_{X}^{\partial}\left(\left(x,\beta\right)\right)=x$. Even if $X$
is connected, $\partial X$ may be disconnected and $i_{X}^{\partial}$
may not be injective. Sometimes we abbreviate $i^{\partial}=i_{X}^{\partial}$. 

A strongly smooth map $f:X\to Y$ between manifolds with corners is
\emph{a submersion }if, whenever $x$ of depth $k$ maps to $y=f\left(x\right)$
of depth $l$, both $df|_{x}:T_{x}X\to T_{y}Y$ and $df|_{x}:T_{x}S^{k}\left(X\right)\to T_{y}S^{l}\left(Y\right)$
are surjective (see \cite[Definition 3.2]{joyce-fibered}; beware
that a ``smooth map'' there is what we call a strongly smooth map,
see \cite[Remark 2.4,(iii)]{joyce-generalized}). We say a map $f:X\to Y$
is \emph{perfectly simple }if it is simple and maps points of depth
$k$ to points of depth $k$, and is \emph{étale }if it is a local
diffeomorphism.

If $X$ is a manifold with corners its tangent bundle $TX$ is defined
in the obvious way. In addition, one can consider the \emph{b-tangent
bundle} $^{b}TX$ . It is a vector bundle on $X$ whose sections can
be identified with sections $v\in C^{\infty}\left(TX\right)$ such
that $v|_{S^{k}\left(X\right)}$ is tangent to $S^{k}\left(X\right)$
for all $k$ (cf. \cite[Definition 2.15]{joyce-generalized}). If
$f:X\to Y$ is an interior map of orbifolds with corners, there's
an induced map $^{b}df:^{b}TX\to^{b}TY$. Two interior maps $f:X\to Z$
and $g:Y\to Z$ are called \emph{b-transverse }if for any $x\in S^{j}\left(X\right),\,y\in S^{k}\left(Y\right)$
such that $f\left(x\right)=g\left(y\right)=z$, the map 
\[
^{b}df\oplus^{b}dg:{}^{b}T_{x}X\oplus^{b}T_{y}Y\to^{b}T_{z}Z
\]
is surjective. 
\begin{rem}
\label{rem:easy b-transversality}In case $\partial Z=\emptyset$,
$f,g$ are b-transverse if and only if for every $x\in S^{j}\left(X\right),y\in S^{k}\left(Y\right)$
with $f\left(x\right)=g\left(y\right)=z$ the map
\[
df|_{TS^{k}\left(X\right)}\oplus dg|_{TS^{l}\left(Y\right)}:TS^{k}\left(X\right)\oplus TS^{l}\left(Y\right)\to T_{z}Z
\]
is surjective.
\end{rem}
\begin{lem}
\label{prop:fibered product in manc}Let $X,Y,Z$ be manifolds with
corners and let $f:X\to Z$ and $g:Y\to Z$ be continuous. Consider
the topological fiber product 
\[
P=X\fibp{f}{g}Y=\left\{ \left(x,y\right)\in X\times Y|f\left(x\right)=g\left(y\right)\right\} .
\]
Suppose at least one of the following conditions holds.

(i) $f$ is a b-normal submersion and $g$ is strongly smooth and
interior,

(ii) $f$ is étale, $g$ is a smooth map,

(iii) $f$ is a b-submersion, $g$ is perfectly simple, or

(iv) $\partial Z=\emptyset$, $f,g$ are b-transverse and smooth.

Then $P$ admits a unique structure of a manifold with corners making
it the fiber product in $\manc$, and we have
\begin{equation}
C_{i}\left(W\right)=\coprod_{j,k,l\geq0;i=j+k-l}C_{j}^{l}\left(X\right)\times_{C_{l}\left(Z\right)}C_{k}^{l}\left(Y\right)\label{eq:corners as fibered prods}
\end{equation}
where $C_{j}^{l}\left(X\right)=C_{j}\left(X\right)\cap C\left(f\right)^{-1}\left(C_{l}\left(Z\right)\right)$
and $C_{k}^{l}\left(Y\right)=C_{k}\left(Y\right)\cap C\left(g\right)^{-1}\left(C_{l}\left(Z\right)\right)$,
and the fiber product is taken over $C\left(f\right),C\left(g\right)$. 

Moreover, if $X\xrightarrow{f}Z$ (respectively, $Y\xrightarrow{g}Z$)
is b-normal then so is $P\xrightarrow{f'}Y$ (resp., $P\xrightarrow{g'}X$).
\end{lem}
\begin{proof}
Let $\text{\textbf{Man}}^{gc}$ denote the category of manifolds with
generalized corners with smooth maps (cf. \cite{joyce-generalized}).
This category contains $\manc$ as a full subcategory. In cases (i),
(iii) and (iv) the fiber product exists in $\text{\textbf{Man}}^{gc}$
as an embedded submanifold of $X\times Y$, and (\ref{eq:corners as fibered prods})
holds, by \cite[Proposition 4.25, Theorem 4.28]{joyce-generalized}.
Since the structure of an embedded submanifold is unique if it exists
\cite[Corollary 4.12]{joyce-generalized}, it suffices to check that
the fiber product is in fact a manifold with ordinary corners. In
cases (i) and (iv) this follows from \cite[Theorem 6.4]{joyce-fibered},
and in case (iii) this follows from \cite[Theorem B]{generalized2ordinary}.

Case (ii) is proven by a simple direct argument.

The last statement follows from \cite[Proposition 2.11(c)]{joyce-generalized}:
if $f$ is b-normal then $l\leq j$ in (\ref{eq:corners as fibered prods}),
which implies $i\leq k$, so $f'$ is b-normal.
\end{proof}

In what follows the discussion diverges from \cite{joyce-generalized}
(see more specifically $\S$4.2 there). More precisely we introduce
a stronger notion of a closed immersion, that has the implicit function
theorem built into it. This is the only kind of closed immersion that
we need to consider, and makes the discussion considerably simpler.
\begin{defn}
\label{def:cl im}A map $f:X\to Y$ of manifolds wtih corners is called
a \emph{closed immersion} if for every $p\in X$ there exists an open
neighborhood $p\in U\subset X$, an open neighborhood $f\left(U\right)\subset V\subset Y$,
and a strongly smooth submersion $h:V\to\rr^{N}$ for some integer
$N\geq0$ such that the following square is cartesian 
\[
\xymatrix{U\ar[r]^{f|_{U}}\ar[d] & V\ar[d]^{h}\\
0\ar[r] & \rr^{N}
}
\]
(it follows that $N=\dim Y-\dim X$). Note the fiber product exists
by Lemma \ref{prop:fibered product in manc} since $h$ is (vacuously)
b-normal, and $0\to\rr^{N}$ is strongly smooth and interior.
\end{defn}
\begin{rem}
\label{rem:b-submersive is enough}Any b-submersion to a manifold
without boundary is automatically a strongly smooth submersion, so
in Definition \ref{def:cl im} it suffices to assume that $h$ is
a b-submersion. 
\end{rem}

\begin{defn}
A map $f:X\to Y$ of manifolds with corners is called \emph{a closed
embedding }if it is a closed immersion, has a closed image, and induces
a homeomorphism on its image.
\end{defn}

\begin{defn}
A map $f:X\to Y$ of manifolds with corners is an \emph{open embedding}
if it is étale and injective.
\end{defn}
\begin{lem}
If $i:X\to Y$ and $f:\mathcal{W}\to\yy$ are smooth maps of manifolds
with corners, with $i$ either a closed or an open embedding, and
if $f\left(W\right)\subset i\left(X\right)$, then there is a unique
smooth map $g:W\to X$ with $f=i\circ g$. If we assume in addition
that $f$ is also a closed or an open embedding, and that $f\left(W\right)=i\left(X\right)$,
then $g$ is a diffeomorphism.
\end{lem}
\begin{proof}
A closed or open embedding is an embedding in the sense of \cite{joyce-generalized},
so this is a special case of Corollary 4.11 \emph{ibid.}. The last
statement is immediate.
\end{proof}
\begin{defn}
(a) Let $f:X\to Y$ be a map of manifolds with corners. We say $f$
is \emph{horizontally submersive }if for every $\tilde{x}\in X$ the
germ $f_{\tilde{x}}$ is isomorphic to the projection $\rr_{k}^{n}\to\rr_{k'}^{n'}$,
\[
\left(x_{1},...,x_{n}\right)\mapsto\left(x_{1},...,x_{k'},x_{k+1},...,x_{k+n'-k'}\right).
\]

(b) Let $f:X\to Y$ be a b-normal map. We call
\[
C_{k}^{hor}\left(X\right):=\left(C\left(f\right)^{-1}\left(C_{0}\left(Y\right)\right)\cap C_{k}\left(X\right)\right)
\]
the \emph{horizontal k-corners }of $X$ with respect to $f$. 
\end{defn}
\begin{lem}
A map $f:X\to Y$ is horizontally submersive if and only if it is
b-normal and the induced map $C_{k}^{hor}\left(X\right)\xrightarrow{C\left(f\right)}Y$is
a submersion for every $k$; that is, 
\[
T_{x}C_{k}^{hor}\left(X\right)\xrightarrow{dC\left(f\right)}T_{y}Y
\]
is surjective for all $x\in C_{k}^{hor}\left(X\right)$.
\end{lem}
\begin{proof}
The ``only if'' part is straightforward. Suppose $f$ is b-normal
and $C\left(f\right)|_{C_{k}^{hor}\left(X\right)}$ is a submersion
for all $k$, and let us prove it is horizontally submersive. We have
$C_{0}^{hor}\left(X\right)=C_{0}\left(X\right)$. Fix some $\tilde{x}$
of depth $k$ and suppose $\tilde{y}=f\left(\tilde{x}\right)$ has
depth $k'$. Since $f$ is b-normal and the induced map $f:C_{0}\left(X\right)\to Y$
is a submersion, we find that there's an injective map $\left\{ 1,...,k'\right\} \xrightarrow{\iota}\left\{ 1,...,k\right\} $
such that the germ $f_{\tilde{x}}$ is isomorphic to the map 
\[
\left(x_{1},...,x_{n}\right)\mapsto\left(y_{1},...,y_{n'}\right)
\]
where for each $1\leq j\leq k'$
\[
y_{j}=Y_{j}\left(x_{1},...,x_{n}\right)\cdot x_{\iota\left(j\right)}
\]
for some $\left(0,\infty\right)$-valued smooth functions $Y_{j}$
defined in a neighborhood of $\tilde{x}$.

We also assumed that $f_{\tilde{x}}|_{\left\{ x_{1}=\cdots=x_{k}=0\right\} }$
is a submersion. It follows that there's some subset of the coordinates
$\left(x_{a_{1}},...,x_{a_{n'}}\right)$ such that 
\begin{enumerate}
\item $a_{j}=\iota\left(j\right)$ for $1\leq j\leq k'$, 
\item $k+1\leq a_{j}\leq n$ for $\left(k'+1\right)\leq j\leq n'$, and 
\item $\left(\frac{\partial f_{j}}{\partial x_{a_{j'}}}\right)_{1\leq j,j'\leq n'}$
is invertible.
\end{enumerate}
Let $\left\{ b_{1},...,b_{n''}\right\} $ be some enumeration of the
complement of ${\left\{ a_{1},...,a_{n'}\right\} \subset\left\{ 1,...,n\right\} }$.
The map $\left(f_{1},...,f_{n'},x_{b_{1}},...,x_{b_{n''}}\right)$
is seen to be a local diffeomorphism, and the result follows.
\end{proof}

Suppose now $X,Y$ are manifolds with corners, $f$ is horizontally
submersive with oriented fibers, and let $\omega$ be a compactly
supported differential form on $X$. In this case we can define $f_{*}\omega$
by integration along the fiber.

\subsection{Orbifolds with corners}
\begin{defn}
A groupoid $\left(G_{0},G_{1},s,t,e,i,m\right)$ is a category where
every arrow is invertible. Namely, $G_{0}$ is a class of \emph{points}
and $G_{1}$ is a class of \emph{arrows}. The\emph{ }maps $s,t:G_{1}\to G_{0}$
take an arrow to its \emph{source }and \emph{target }objects, respectively.
The \emph{composition} map ${m:\left\{ \left(f,g\right)\in G_{1}\times G_{1}|t\left(f\right)=s\left(g\right)\right\} \to G_{1}}$
takes a pair of composable arrows to their composition. The \emph{identity
map} $e:G_{0}\to G_{1}$ takes an object to the identity arrow and
the \emph{inverse map} $i:G_{1}\to G_{1}$ takes an arrow to its inverse. 
\end{defn}
The equivalence classes of the equivalence relation ${\im\left(s\times t\right)\subset G_{0}\times G_{0}}$
are called \emph{the orbits} of the groupoid; the class of all orbits
is denoted $G_{0}/G_{1}$. We will use different notations for groupoids,
depending on how much of the structure we want to label:
\[
\left(G_{0},G_{1},s,t,e,i,m\right)=G_{\bullet}=G_{1}\overset{s,t}{\rightrightarrows}G_{0}.
\]

\begin{defn}
\label{def:groupoids in ManC}A groupoid $\left(X_{0},X_{1},s,t,e,i,m\right)$
will be called \emph{étale }if $X_{0},X_{1}$ are objects of $\manc$,
and the maps $s,t,e,i,m$ are all étale (in fact, it suffices to require
that $s:X_{1}\to X_{0}$ is étale). An étale groupoid will be called
\emph{proper }if the map $s\times t:X_{1}\to X_{0}\times X_{0}$ is
proper. We will mostly be interested in proper étale groupoids, or
PEG's for short.

Let $X_{\bullet}$ be a PEG. The set of orbits $X_{0}/X_{1}$, taken
with the quotient topology, forms a locally compact Hausdorff space.
$X_{\bullet}$ is called \emph{compact }if $X_{0}/X_{1}$ is compact. 
\end{defn}
Let $X_{\bullet},Y_{\bullet}$ be two PEG's. A\emph{ smooth functor}
$X_{\bullet}\xrightarrow{F_{\bullet}}Y_{\bullet}$ consists of a pair
of  smooth maps $F_{0}:X_{0}\to Y_{0}$ and $F_{1}:X_{1}\to Y_{1}$
which is a functor between the underlying categories. If $F_{\bullet},G_{\bullet}:X_{\bullet}\to Y_{\bullet}$
are two functors a\emph{ smooth transformation }$\alpha:F_{\bullet}\Rightarrow G_{\bullet}$\textbf{
}is a smooth map $X_{0}\to Y_{1}$ which is a natural transformation
between the underlying functors. In this way we obtain a bicategory
(see \cite{benabou}) $\mathbf{PEG}$, whose objects, or \emph{0-cells},
are proper étale groupoids, morphisms (or \emph{1-cells}) are smooth
functors, and 2-cells are natural transformations. A \emph{refinement
}$R_{\bullet}:X_{\bullet}\to X'_{\bullet}$ is a smooth functor which
is an equivalence of categories and such that $R_{0}$ (hence also
$R_{1}$) is an étale map.
\begin{lem}
As a subset of the 1-cells of $\mathbf{PEG}$ the refinements admit
a right calculus of fractions, in the sense of \cite[\S 2.1]{pronk}.
\end{lem}
\begin{proof}
We use the notation \emph{ibid}. BF1, BF2 and BF5 are straightforward.
To establish BF3, use the weak fiber product (the construction of
the weak fibered product in $\text{\textbf{PEG}}$ is reviewed in
Lemma \ref{lem:fibp in orb} below). We prove BF4. Suppose $f_{\bullet},g_{\bullet}:X_{\bullet}\to Y_{\bullet}$
are smooth functors and $w_{\bullet}:\left(Y_{1}\overset{s,t}{\rightrightarrows}Y_{0}\right)\to\left(Y_{1}'\overset{s',t'}{\rightrightarrows}Y_{0}'\right)$
is a refinement, and $\alpha:w_{\bullet}\circ f_{\bullet}\Rightarrow w_{\bullet}\circ g_{\bullet}$
is a smooth transformation (here we use lowercase letters to denote
functors, to keep close to the notation in \cite{pronk}). Since the
following square is cartesian 
\[
\xymatrix{Y_{1}\ar[r]^{w_{1}}\ar[d]_{s\times t} & Y_{1}'\ar[d]^{s'\times t'}\\
Y_{0}\times Y_{0}\ar[r]_{w_{0}\times w_{0}} & Y_{0}'\times Y_{0}'
}
\]
the maps $\alpha:X_{0}\to Y_{1}$ and $f_{0}\times g_{0}:X_{0}\to Y_{0}\times Y_{0}$
define the desired $\beta:X_{0}\to Y_{1}$, with $v=\id_{X_{\bullet}}$.
The second requirement holds since all 2-cells are invertible. For
the final requirement, take $u=v',u'=\id$ and $\epsilon=\id$.
\end{proof}
We define the category $\mathbf{Orb}$ of orbifolds (always with corners,
unless specifically mentioned otherwise) to be the 2-localization
of $\mathbf{PEG}$ by the refinements. We usually denote orbifolds
by calligraphic letters $\xx,\yy,\mm$... They are given by proper
étale groupoids. Maps $\xx\to\yy$ are given by fractions $F_{\bullet}|R_{\bullet}$
with $X_{\bullet}\xleftarrow{R_{\bullet}}X_{\bullet}'$ a refinement
and $X'_{\bullet}\xrightarrow{F_{\bullet}}Y_{\bullet}$ a smooth functor.
We refer the reader to \cite{pronk} for further details, including
the definition of the 2-cells, the composition operations, etc.
\begin{rem}
\label{rem:complex orbifolds}We will occasionally consider other
categories of orbifolds. First, there's the category of orbifolds
without boundary $\text{\textbf{Orb}}_{\partial=\emptyset}$. This
can be realized simply as the 2-full bicategory of $\text{\textbf{Orb}}$
spanned by all objects $\xx$ with $\partial\xx=\emptyset$. We will
also encounter the category $\text{\textbf{Orb}}_{\cc}$ of \emph{complex
orbifolds.} To construct it, we begin with the bicategory $\text{\textbf{PEG}}_{\cc}$
whose objects are groupoids $M_{1}\rightrightarrows M_{0}$ where
$M_{i}$, $i=0,1$, is a complex manifold, the structure maps $s,t,e,i,m$
are local biholomorphisms, and $s\times t$ is proper. 1-cells and
2-cells in $\text{\textbf{PEG}}_{\cc}$ are given by holomorphic functors
and holomorphic natural transformations, respectively. To obtain $\text{\textbf{Orb}}_{\cc}$
we invert \emph{holomorphic refinements,} that is, equivalences $\left(R_{0},R_{1}\right)$
where $R_{0}$ is a local biholomorphism. There's an obvious way to
extend $\text{\textbf{Orb}}_{\cc}$ to allow also \emph{anti}-holomorphic
morphisms, where the category of \emph{antiholomorphic }morphisms
$\left(\xx=X_{1}\rightrightarrows X_{0}\right)\to\yy$ is, by definition,
equal to the category of morphisms $\overline{\xx}\to\yy$ in $\text{\textbf{Orb}}_{\cc}$
where $\overline{\xx}=\overline{X}_{1}\rightrightarrows\overline{X}_{0}$.

There are obvious bifunctors 
\[
\text{\textbf{Orb}}_{\cc}\to\text{\textbf{Orb}}_{\cc}^{+}\to\text{\textbf{Orb}}_{\partial=\emptyset}\to\text{\textbf{Orb}}.
\]
\end{rem}
\begin{defn}
\label{def:properties of orbi-maps}We say $f$ is\emph{ strongly-smooth,
étale, interior, b-normal, submersive, b-submersive, horizontally
submersive, simple }or\emph{ perfectly simple }if $F_{0}$ has the
corresponding property as a map of manifolds with corners. It is
easy to check that these properties are preserved by 2-cells (and
thus are properties of the homotopy class of $f$). The map $f$ is
called a\emph{ b-fibration} if it is b-normal and b-submersive (cf.
\cite[Definition 4.3]{joyce-generalized}).

For $i=1,2$ let $f^{i}=F^{i}|R^{i}:\xx^{i}\to\yy$ be an interior
map. We say $f^{1}$ and $f^{2}$ are \emph{b-transverse }if $F_{0}^{1},F_{0}^{2}$
are b-transverse (as maps of manifolds with corners). 
\end{defn}
An equivalence in $\mathbf{Orb}$ is called \emph{a diffeomorphism.}
We say $f=F|R:\xx\to\yy$ is \emph{full}, \emph{essentially surjective},
or \emph{faithful} if $F$ is full, essentially surjective, or faithful,
respectively.

If $\xx=X_{1}\rightrightarrows X_{0}$ is an orbifold with corners,
$\partial\xx=\partial X_{1}\rightrightarrows\partial X_{0}$ is naturally
an orbifold with corners and the smooth functor $\left(i_{X_{1}}^{\partial},i_{X_{0}}^{\partial}\right)$
induces a map $i_{\xx}^{\partial}:\partial\xx\to\xx$. We denote 
\[
i_{\xx}^{\partial^{c}}:=i_{\xx}^{\partial}\circ i_{\partial\xx}^{\partial}\circ\cdots\circ i_{\partial^{c-1}\xx}^{\partial}:\partial^{c}\xx\to\xx.
\]
 Since the maps $s,t,e,i,m$ are étale, they preserve the depth and
we obtain orbifolds with corners
\[
C_{k}\left(\xx\right)=C_{k}\left(X_{1}\right)\rightrightarrows C_{k}\left(X_{0}\right)
\]
for all $k$. A \emph{local orbifold with corners }$\xx=\coprod\xx_{n}$
(or just an $l$-orbifold) is a disjoint union of orbifolds with corners
with $\dim\xx_{n}=n$. It is obvious how to turn this into a category
and extend the definitions of various types of maps to this situation.
If $\xx$ is an orbifold with corners, we can consider $C\left(\xx\right)=\coprod_{k\geq0}C_{k}\left(\xx\right)$
as an l-orbifold. A smooth map $f:\xx\to\yy$ induces an interior
map $C\left(f\right):C\left(\xx\right)\to C\left(\yy\right)$.

We turn to a discussion of the weak fibered product in $\text{\textbf{Orb}}$. 
\begin{lem}
\label{lem:fibp in orb}Let
\begin{equation}
f:\xx\xleftarrow{R}\xx'\xrightarrow{F}\zc\text{ and }g:\yy\xleftarrow{S}\yy'\xrightarrow{G}\zc\label{eq:f and g fracs}
\end{equation}
be two 1-cells in $\text{\textbf{Orb}}$. Suppose at least one of
the following conditions holds.

(i) F is a b-normal submersion and $G$ is strongly smooth and interior,

(ii) $F$ is étale, $G$ is a smooth map,

(iii) $F$ is a b-submersion, $G$ is perfectly simple, or

(iv) $\partial\zc=\emptyset$, $F$ and $G$ are b-transverse (see
Remark \ref{rem:easy b-transversality} for an equivalent condition)
and smooth. 

Then 

(a) The weak fiber product $\mathcal{P}=\xx\fibp{f}{g}\yy$ exists
in $\text{\textbf{Orb}}$. In fact, we can take 
\[
\mathcal{P}=\xx'\fibp{F}{G}\yy'
\]
the weak fiber product in $\text{\textbf{PEG}}$, given by the groupoid
$P_{1}\rightrightarrows P_{0}$ where 
\begin{eqnarray*}
P_{0}=X_{0}'\fibp{F_{0}}{s}Z_{1}\fibp{t}{G_{0}}Y_{0}',\\
P_{1}=X_{1}'\fibp{s\circ F_{1}}{s}Z_{1}\fibp{t\circ G_{1}}{s}Y_{1}'.
\end{eqnarray*}
Here an element of $P_{1}$ specifies the three solid arrows in the
diagram below, 
\[
\xymatrix{x^{1}\ar[d]_{a} & F_{0}\left(x^{1}\right)\ar@{-->}[d]_{F_{1}\left(a\right)}\ar[r] & G_{0}\left(y^{1}\right)\ar@{-->}[d]^{G_{1}\left(b\right)} & y^{1}\ar[d]^{b}\\
x^{2} & F_{0}\left(x^{2}\right)\ar@{-->}[r] & G_{0}\left(y^{2}\right) & y^{2}
}
.
\]
The horizontal dashed arrow is uniquely determined by requiring the
square to be commutative; $s,t:P_{1}\to P_{0}$ are the projections
on the top and bottom rows of the diagram, respectively, and the
other structure maps are computed similarly.

(b) We have
\begin{equation}
C_{i}\left(\mathcal{P}\right)=\coprod_{j,k,l\geq0;i=j+k-l}C_{j}^{l}\left(\xx\right)\times_{C_{l}\left(Z\right)}C_{k}^{l}\left(\yy\right)\label{eq:corners as fibp in orb}
\end{equation}
where $C_{j}^{l}\left(\xx\right)=C_{j}\left(\xx\right)\cap C\left(f\right)^{-1}\left(C_{l}\left(\zc\right)\right)$
and $C_{k}^{l}\left(\yy\right)=C_{k}\left(\yy\right)\cap C\left(g\right)^{-1}\left(C_{l}\left(\zc\right)\right)$,
and the weak fiber product is taken over $C\left(f\right),C\left(g\right)$. 

(c) If we assume, in addition, that $\xx\xrightarrow{f}\zc$ (respectively,
$\yy\xrightarrow{g}\zc$) is b-normal, then so is $\mathcal{P}\xrightarrow{f'}\yy$
(resp., $\mathcal{P}\xrightarrow{g'}\xx$).
\end{lem}
\begin{proof}
It is well-known (and easy to verify) that the weak fibered product
\[
\mathcal{P}=\xx'\fibp{F}{G}\yy'
\]
in the bicategory of proper étale groupoids \emph{in topological spaces}
is represented by $P_{1}\rightrightarrows P_{0}$ as described above
(see \cite{moerdijk-classifying-groupoids}). Using Proposition \ref{prop:fibered product in manc}
it is not hard to show that $P_{1},P_{0}$ are smooth manifolds with
corners, that the structure maps are étale (i.e., a local \emph{diffeo}morphism),
and that $\mathcal{P}$ represents the weak fibered product in $\text{\textbf{PEG}}$,
the category of proper étale groupoids \emph{in} \emph{manifolds with
corners}. 

It then follows from a result of Tommasini \cite[Corollary 0.3]{wfibp}
that 
\[
\mathcal{P}=\xx\fibp{f}{g}\yy,
\]
the weak fiber product of $f$ and $g$ \emph{in $\text{\textbf{Orb}}$. }

Claims (b) and (c) are straightforward, again using Proposition \ref{prop:fibered product in manc}.
\end{proof}

\begin{defn}
A map $F|R:\xx\to\yy$ of orbifolds with corners is a \emph{closed
immersion} if $F_{0}$ is a closed immersion. In this case, the same
holds for any map homotopic to $F|R$.
\end{defn}

A manifold with corners $M$ specifies an orbifold $\underline{M}=M\rightrightarrows M$
with only identity morphisms, and this extends to a 2-fully-faithful
pseudofunctor $\manc\to\text{\textbf{Orb}}$ (namely, it restricts
to an equivalence $\manc\left(X,Y\right)\simeq\text{\textbf{Orb}}\left(\underline{X},\underline{Y}\right)$
for any pair $X,Y$ of objects of $\manc$). We say an orbifold ``is''
a manifold with corners if it is in the essential image of this functor.
\begin{defn}
Let $\xx$ be an orbifold with corners. \emph{An atlas} \emph{for
$\xx$ }is a map $p:\underline{M}\to\xx$ where $M$ is some manifold
with corners, such that for any other map $f:\underline{N}\to\xx$
from a manifold with corners, $\underline{M}\times_{\xx}\underline{N}$
is a manifold with corners and the projection $\underline{M}\times_{\xx}\underline{N}\overset{p'}{\longrightarrow}\underline{N}$
is étale and surjective (as a map of $\manc$).
\end{defn}
The obvious map $\underline{X_{0}}\to\left(X_{1}\rightrightarrows X_{0}\right)$
is an atlas. Conversely, any atlas $\underline{M}\to\xx$ defines
an orbifold equivalent to $\xx$, whose objects are $\underline{M}$
and morphisms are $\underline{M}\times_{\xx}\underline{M}$.
\begin{defn}
\label{def:closed embedding}A map $f:\xx\to\yy$ of orbifolds with
corners is a \emph{closed} (respectively, \emph{open}) \emph{embedding}
if for some (hence any) atlas $p:\underline{M}\to\yy$, the 2-pullback
$\underline{M}\fibp{p}{f}\xx$ is a manifold with corners  and the
map $\underline{M}\fibp{p}{f}\xx\to\underline{M}$ is a closed (resp.
open) embedding of manifolds with corners.
\end{defn}
If $f:\xx\to\yy$ is a closed embedding we may refer to $\xx$ as
a \emph{suborbifold} of $\yy$.

The notion of a sheaf on an orbifold $\xx$ is the same as the notion
of a sheaf on the underlying topological orbifold, see Moerdijk and
Pronk \cite{moerdijk-classifying-groupoids,pronk} for a comprehensive
treatment. A vector bundle $E$ on an orbifold with corners $\xx=X_{1}\overset{s,t}{\rightrightarrows}X_{0}$
is given by $\left(E_{0},\phi\right)$ where $E_{0}$ is a smooth
vector bundle on $X_{0}$ and 
\[
\phi:s^{*}E_{0}\to t^{*}E_{0}
\]
is an isomorphism satisfying some obvious compatibility requirements
with the groupoid structure. The sections of $\left(E_{0},\phi\right)$
form a sheaf over $\xx$. An important example of a vector bundle
is the tangent bundle, $T\xx=\left(TX_{0},dt\circ ds^{-1}\right)$,
whose sections are vector fields on $\xx$. A \emph{local system }on
an orbifold $\xx$ is a sheaf which is locally isomorphic to the constant
sheaf $\underline{\zz}$. We extend the conventions set forth in \cite[\S1.1, \S 6.1]{twA8}
to proper étale groupoids with corners in the obvious way\footnote{Note there we had to work with $\cc$-valued local systems, but for
the purposes of this paper we can work with $\zz$-valued local systems.}. In particular, for every vector bundle $E$ on $\xx$ there's a
local system $\Or\left(E\right)$ on $\xx$. The \emph{orientation
local system }of $\xx$ is $\Or\left(T\xx\right)$. We have a local
system isomorphism 
\begin{equation}
\iota_{\xx}^{\partial}:\Or\left(T\partial X{}_{\bullet}\right)\to\Or\left(TX_{\bullet}\right)\label{eq:boundary ls iso}
\end{equation}
lying over $i_{X_{\bullet}}^{\partial}:\partial X_{\bullet}\hookrightarrow X_{\bullet}$,
defined by appending the outward normal to the boundary at the beginning
of the oriented base for $T\partial X_{\bullet}$. Given a short exact
sequence of vector bundles
\[
0\to E_{1}\xrightarrow{f}E\xrightarrow{q}E_{2}\to0
\]
on $\xx$, we obtain a local system isomorphism 
\begin{equation}
\Or\left(E_{1}\right)\otimes\Or\left(E_{2}\right)\to\Or\left(E\right),\label{eq:ls map from ses}
\end{equation}
which, using oriented bases to represent orientation, can be expressed
by
\[
\left[e_{1}^{1},...,e_{1}^{n_{1}}\right]\otimes\left[e_{2}^{1},...,e_{2}^{n_{2}}\right]\mapsto\left[f\left(e_{1}^{1}\right),...,f\left(e_{1}^{n_{1}}\right),g\left(e_{2}^{1}\right),...,g\left(e_{2}^{n_{2}}\right)\right]
\]
where $g:E_{2}\to E$ is any local section of $q$.

Maps of local systems are always assumed to be cartesian, so to specify
a local system map ${\lc_{1}\xrightarrow{\ff}\lc_{2}}$ over ${\xx_{1}\xrightarrow{f}\xx_{2}}$
is equivalent to giving an isomorphism $\lc_{1}\to f^{-1}\lc_{2}$.
\begin{lem}
\label{lem:ls extensions}Let $\xx$ be an orbifold with corners.
We denote by $\mathring{\xx}:=S^{0}\left(\xx\right)$ the orbifold
(without boundary or corners) consisting of points of depth zero,
and by $j:\mathring{\xx}\hookrightarrow\xx$ the inclusion. 

(a) The pushforward and inverse image functors $j_{*},j^{-1}$ form
an adjoint equivalence of groupoids between local systems on $\mathring{\xx}$
and local systems on $\xx$.

(b) $Or\left(dj\right):Or\left(T\mathring{\xx}\right)\to j^{-1}Or\left(T\xx\right)$
is an isomorphism.

Let $f:\xx\to\yy$ be a b-normal map of orbifolds with corners. 

(c) There exists a unique map $\mathring{f}:\mathring{\xx}\to\mathring{\yy}$
with $f\circ j_{\xx}=j_{\yy}\circ\mathring{f}$.

Let $\lc$ be a local system on $\xx$ and let $\lc'$ be a local
system on $\yy$, and denote by $\mathring{\lc}=j_{\xx}^{-1}\lc$,
$\mathring{\lc}':=j_{\yy}^{-1}\lc'$ their restrictions to $\mathring{\xx},\mathring{\yy}$,
respectively. Define a map taking a map of sheaves $\ff:\lc\to\lc'$
over $f$ to the map $\mathring{\ff}:\mathring{\lc}\to\mathring{\lc}'$
over $\mathring{f}$ given by the composition 
\[
j_{\xx}^{-1}\lc\overset{j_{\xx}^{-1}\ff}{\longrightarrow}j_{\xx}^{-1}f^{-1}\lc'\simeq\mathring{f}^{-1}j_{\yy}^{-1}\lc'.
\]

(d) $\ff\mapsto\mathring{\ff}$ is a bijection
\[
\left\{ \mbox{maps }\ff:\lc\to\lc'\mbox{ over }f\right\} \simeq\left\{ \mbox{maps }\mathring{\ff}:\mathring{\lc}\to\mathring{\lc}'\mbox{ over }\mathring{f}\right\} .
\]
and together with $\lc\mapsto\mathring{\lc}$ forms a functor from
the category of sheaves (respectively, local systems) over orbifolds
with corners with b-normal maps to the category of sheaves (resp.
local systems) over orbifolds without boundary.
\end{lem}
\begin{proof}
Straightforward.
\end{proof}
Let $\xx$ be an orbifold with corners and $\lc$ a local system on
$\xx$. We define the complex of \emph{differential forms on $\xx$
with values in $\lc$}
\[
\Omega\left(\xx;\lc\right)=\Gamma\left(C^{\infty}\left(\bigwedge T\xx\right)\otimes_{\zz}\lc\right)
\]
as the global sections of the sheaf of sections of the vector bundle
$\bigwedge T\xx$, twisted by $\lc$. 

Suppose $\xx,\yy$ are compact orbifolds with corners, $\mathcal{K},\lc$
are local systems on $\xx$ and on $\lc$, respectively, and $f:\left(\xx,\mathcal{K}\right)\to\left(\yy,\lc\right)$
is an \emph{oriented }map, which means it is a map of local systems
$\mathcal{K}\to\lc$ lying over a smooth map of orbifolds with corners
$\xx\to\yy$. We have a pullback operation

\begin{equation}
\Omega\left(\yy;\lc\right)\xrightarrow{f^{*}}\Omega\left(\xx;\mathcal{K}\right).\label{eq:pullback}
\end{equation}
If, in addition, we assume that $f$ is horizontally submersive, then
there's a pushforward operation

\begin{equation}
\Omega\left(\xx;\kk\otimes Or\left(T\xx\right)^{\vee}\right)\xrightarrow{f_{*}}\Omega\left(\yy;\lc\otimes Or\left(T\yy\right)^{\vee}\right).\label{eq:pushforward}
\end{equation}
We now sketch how these operations are constructed. Define the complex
of compactly supported differential forms\emph{ }on $\xx$ by
\[
\Omega_{c}\left(\xx;\lc\right):=\coker\left(t_{*}-s_{*}:\Omega_{c}\left(X_{1};s^{*}\lc_{0}\right)\to\Omega_{c}\left(X_{0};\lc_{0}\right)\right),
\]
where on the right hand side, $\Omega_{c}$ denotes the usual complex
of compactly supported forms on a manifold with corners. In case
$f=\left(F_{0},F_{1}\right):\xx\to\yy$ is a smooth functor, $F_{0}^{*}$
induces a pullback map (\ref{eq:pullback}) and (if $f$ is a horizontally
submersive) $\left(F_{0}\right)_{*}$ induces a pushforward map of
compactly supported forms, 
\begin{equation}
\Omega_{c}\left(\xx;\kk\otimes Or\left(T\xx\right)^{\vee}\right)\xrightarrow{f_{*}}\Omega_{c}\left(\yy;\lc\otimes Or\left(T\yy\right)^{\vee}\right).\label{eq:pushforward compact support}
\end{equation}
In defining the operations $F_{0}^{*}$ and $\left(F_{0}\right)_{*}$
(for forms on manifolds with corners) we follow the orientation conventions
in \cite{twA8}. A \emph{partition of unity} \emph{for $\xx$} is
a smooth map\emph{ }$\rho:X_{0}\to\left[0,1\right]$ such that $\operatorname{supp}\left(s^{*}\rho\right)\cap t^{-1}\left(K\right)$
is compact for every compact subset $K\subset X_{0}$ and $t_{*}s^{*}\rho\equiv1$
(the fiber of $t$ is discrete, hence canonically oriented). Partitions
of unity always exist; since $\xx$ is assumed to be compact we can
require that $\rho$ has compact support in $X_{0}$, and use this
to construct an isomorphism
\begin{equation}
\Omega\left(\xx;\lc\right)\simeq\Omega_{c}\left(\xx;\lc\right),\label{eq:Poincare duality}
\end{equation}
see Behrend \cite{behrend}. The isomorphism (\ref{eq:Poincare duality})
allows us to define (\ref{eq:pushforward}) using (\ref{eq:pushforward compact support}).
Now if $f=\xx\xleftarrow{R}\xx'\xrightarrow{F}\yy$ is a general oriented
map, we define (\ref{eq:pullback}) by
\[
f^{*}=R_{*}F^{*},
\]
pulling back along the smooth functor $F$ and then pushing forward
along the refinement $R$ (note $R$ is horizontally submersive since
it is étale; moreover, any refinement defines an equivalence between
the categories of local systems on $\xx$ and on $\xx'$, so orientations
for $f$ are in natural bijection with orientations for $F$). If
$f$ is oriented and horizontally submersive we define the pushforward
(\ref{eq:pushforward}) by 
\[
f_{*}=F_{*}R^{*}.
\]
By construction, the operations (\ref{eq:pullback}, \ref{eq:pushforward})
extend the operations defined in \cite{twA8} for the case $\xx,\yy$
are manifolds, and they satisfy the same relations.

To make the paper more readable, outside of this appendix we will
sometimes abuse notation and refer to maps which have a specified
isomorphism as being equal. For example, if $G$ acts on $\xx$ (see
$\S$\ref{subsec:Group-actions} below) we may write 
\[
g.h.=\left(gh\right).
\]
even though in general the two sides differ by a (specified) 2-cell.
The same goes for orbifolds which are canonically equivalent (that
is, with a given equivalence, or with an equivalence which is specified
up to a unique 2-cell). For example we may write
\[
\left(\mm_{1}\times\mm_{2}\right)\times\mm_{3}=\mm_{1}\times\left(\mm_{2}\times\mm_{3}\right).
\]
When we write $p\in\xx$ we mean $p\in X_{0}$, where $\xx=X_{1}\rightrightarrows X_{0}$.

\subsection{\label{subsec:Hyperplane-Blowup}Hyperplane Blowup}

In this subsection we explain how to blow up an orbifold along a nice
codimension one locus, to obtain an orbifold with corners. This is
an important step in the construction of the moduli spaces of discs
from the moduli spaces of curves (see $\S$\ref{subsec:blowup M tilde}
and the motivating discussion in the introduction). The construction
is carried out in two steps: first, we discuss the hyperplane blowup
of manifolds, and then we extend this to orbifolds.

\subsubsection{Hyperplane blowup of manifolds}
\begin{defn}
\label{def:hyper}(a) Let $h:W\to X$ be a proper closed immersion
between manifolds without boundary. Write $h^{-1}\left(x\right)=\left\{ w_{1},...,w_{r}\right\} $
(this is finite since $h$ is proper), and let $N_{w_{i}}^{\vee}=\ker\left(T_{x}^{\vee}X\xrightarrow{dh|_{w_{i}}^{\vee}}T_{w_{i}}^{\vee}W\right)$
denote the conormal bundle to $h$. We say $h$ has \emph{transversal
self-intersection} \emph{at $x\in X$} if the induced map 
\[
\bigoplus_{i=1}^{r}N_{w_{i}}^{\vee}\to T_{x}^{\vee}X
\]
is injective. We say $h$ has transversal self-intersection\emph{
}if it has transversal self intersection at every $x\in X$. 

(b) Let $h:W\to X$ be a proper closed immersion which has transversal
self intersection. Suppose further that $h$ is \emph{codimension
one}, i.e. $\dim X-\dim W=1$. In this case we call $E=\im h$ a \emph{hyper
subset}, and call $h$ a \emph{hyper map}. Note since the conditions
on $h$ can be checked locally on the codomain $X$, being a hyper
subset is a local property. Moreover, it follows from Proposition
\ref{prop:orthant charts exist} below that the map $h$ is essentially
unique: if $W\xrightarrow{h}X,W'\xrightarrow{h'}X$ are two hyper
maps with $\im h=\im h'$ then there's a unique diffeomorphism $W\xrightarrow{\phi}W'$
such that $h=h'\circ\phi$.

(c) Let $Y\to X$ be a smooth map of manifolds without boundary, and
let $E\subset X$ be a hyper subset. We say $f$ is \emph{multi-transverse
to $E$ }if for some (hence any) hyper map $h$ such that $E=\im h$,
$f$ is transverse to $h$ and the pullback $f^{-1}W\xrightarrow{f^{-1}h}Y$
has transversal self intersection (so in fact, since $f^{-1}h$ is
necessarily a codimension one proper closed immersion, $f^{-1}E\subset Y$
is a hyper subset).
\end{defn}
The following proposition explains the usefulness of these notions
and prepares the ground for the construction of the hyperplane blow
up.
\begin{prop}
\label{prop:orthant charts exist}Let $h:W\to X$ be a hyper map between
manifolds without boundary. 

(a) For every $x\in X$ with $h^{-1}\left(x\right)=\left\{ w_{1},...,w_{r}\right\} $
there exists an open neighborhood $x\in V\subset X$ such that $h^{-1}\left(V\right)=U_{1}\coprod\cdots\coprod U_{r}$
where $w_{i}\in U_{i}\subset W$ is an open neighborhood for $1\leq i\leq r$,
together with charts $V\xrightarrow{\varphi}\rr^{n}$ and $U_{i}\simeq\rr^{n-1}$
so that $h|_{U_{i}}$ corresponds to the map $\rr^{n-1}\to\rr^{n}$
given by 
\[
\left(t_{1},...,t_{n-1}\right)\mapsto\left(t_{1},...,t_{i-1},0,t_{i},...,t_{n-1}\right).
\]
We call the coordinate chart $V\xrightarrow{\varphi}\rr^{n}$ an \emph{orthant
chart for $h$ at $x$.}

(b) Suppose $Y\xrightarrow{f}X$ is multi-transverse to $\im h$
and let $x\in X$ with $\left|h^{-1}\left(x\right)\right|=r$. Then
there exists an open neighborhood $x\in V\subset X$ together with
orthant charts $f^{-1}\left(V\right)\xrightarrow{\varphi}\rr^{m}$
and $V\xrightarrow{\psi}\rr^{n}$ for $f^{-1}h$ and $h$, respectively,
so that $\psi\circ f\circ\varphi^{-1}$ is given by 
\[
\left(t_{1},...,t_{m}\right)\to\left(t_{1},...,t_{r},\phi_{r+1}\left(t_{1},...,t_{m}\right),...,\phi_{n}\left(t_{1},...,t_{m}\right)\right)
\]
for some smooth functions $\phi_{r+1},...,\phi_{n}$.
\end{prop}
\begin{proof}
We prove part (a). Since $h$ is a codimension one closed immersion
for every $1\leq i\leq r$ there exist open neighborhoods $w_{i}\in U_{i}''\subset W$
and $x\in V_{i}\subset X$ and a submersion $h_{i}:V_{i}\to\rr$ such
that the following square is cartesian
\[
\xymatrix{U_{i}''\ar[r]^{h|_{U_{i}''}}\ar[d] & V_{i}\ar[d]^{v_{i}}\\
0\ar[r] & \rr
}
\]
Since $h$ has transversal self-intersection, $dv_{1},...,dv_{r}$
are linearly independent at $x$ and therefore, in a perhaps smaller
open neighborhood $x\in V'\subset\bigcap_{i=1}^{r}V_{i}$ they extend
to a coordinate chart $\left(v_{1},...,v_{r},v_{r+1},...,v_{n}\right):V'\to\rr^{n}$.
Set $U_{i}'=h^{-1}\left(V'\right)\cap U_{i}''$. Uniqueness of the
pullback implies that for $1\leq i\leq r$, 
\[
\left(v_{1}\circ h|_{U_{i}'},...,\widehat{v_{i}\circ h|_{U_{i}'}},...,v_{n}\circ h|_{U_{i}'}\right):U_{i}'\to\rr^{n-1}
\]
is a coordinate chart for $W$ and $h|_{U_{i}'}$ obtains the desired
form in these coordinate systems. Since $h$ is proper, there's an
open neighborhood $x\in V\subset V'$ such that, setting $U_{i}=U_{i}'\cap D^{-1}\left(V\right)$
we have 
\[
h^{-1}\left(V\right)=U_{1}\coprod\cdots\coprod U_{r},
\]
completing the proof of part (a). 

The proof of part (b) is similar.
\end{proof}
\begin{defn}
(a) Let $X$ be a manifold without boundary, let $U\subset X$ be
an open subset. Consider the set of germs of connected components,\emph{
}
\[
I\left(X,U\right)=\bigcup_{x\in X}\left\{ x\right\} \times\lim_{x\in V\subset X}\pi_{0}^{x}\left(V\cap U\right)
\]
where for $V$ an open neighborhood of $x\in X$, $\pi_{0}^{x}\left(V\cap U\right)$
denotes the set of connected components $C\subset V\cap U$ with $x\in\overline{C}$
in the closure. If $V_{1}\subset V_{2}$ are two such neighborhoods,
there's an induced map $\pi_{0}^{x}\left(V_{1}\cap U\right)\to\pi_{0}^{x}\left(V_{2}\cap U\right)$,
and $\lim_{x\in V\subset X}\pi_{0}^{x}\left(V\cap U\right)$ denotes
the \emph{inverse} limit of this system of sets.

(b) If $\left(X_{1},U_{1}\right)\to\left(X_{2},U_{2}\right)$ is a
map of pairs there's an induced map $I\left(X_{1},U_{1}\right)\to I\left(X_{2},U_{2}\right)$
making $I$ a functor; there's an obvious natural transformation $I\left(X,U\right)\to X$.

(c) Let $E\subset X$ be a hyper subset. As a set, \emph{the blow
up of $X$ along $E$ }is given by\emph{ 
\[
B\left(X,E\right)=I\left(X,X\backslash E\right).
\]
}The associated natural transformation is denoted $B\left(X,E\right)\xrightarrow{\beta_{\left(X,E\right)}}X$,
and if $Y\xrightarrow{f}X$ is multi-transverse to $E$ write 
\[
B\left(Y,f^{-1}E\right)\xrightarrow{B\left(f\right)}B\left(X,E\right)
\]
for the induced map.
\end{defn}

\begin{prop}
\label{prop:hyper bu functor ManC}Let $\text{\textbf{Man}}^{+}$
denote the category of \emph{marked manifolds,} whose objects are
pairs $\left(X,E\right)$ where $X$ is a manifold without boundary
and $E$ is a hyper subset of $X$, and where an arrow $\left(X_{1},E_{1}\right)\to\left(X_{2},E_{2}\right)$
is given by a map $X_{1}\xrightarrow{f}X_{2}$ which is multi-transverse
to $E_{2}$ and such that $f^{-1}E_{2}=E_{1}$. Let $\manc_{ps}$denote
the category of manifolds with corners with perfectly simple maps.
Then blowing up gives a faithful functor
\[
B:\text{\textbf{Man}}^{+}\to\manc_{ps}
\]
together with a natural transformation $B\left(X,E\right)\xrightarrow{\beta_{\left(X,E\right)}}X$.

Moreover, if $\left(X_{1},E_{1}\right),\left(X_{2},E_{2}\right)$
are any two objects of $\text{\textbf{Man}}^{+}$, any étale map $f:X_{1}\to X_{2}$
is a morphism of $\text{\textbf{Man}}^{+}$ and $B\left(f\right)$
is also étale in this case.
\end{prop}
The proof of this proposition appears below. The following definition
and lemma characterize the manifold with corners structure on the
blow up. More precisely, $B\left(X,E\right)$ will be equipped with
the unique manifold with corners structure on the set $B\left(X,E\right)$
making the map $\beta_{\left(X,E\right)}$ \emph{rectilinear}:
\begin{defn}
\label{def:rectilinear map}Let $C$ be a manifold with corners, $M$
a manifold without boundary. A map $f:C\to M$ will be called \emph{rectilinear
}if the restriction of $f$ to interior points is an injective map
$\mathring{C}\to M$, and for every $c\in C$ there exist a non-negative
integer $k$ and coordinate charts $U\xrightarrow{\varphi}\rr_{k}^{n},c\in U,\varphi\left(c\right)=0$
and $V\xrightarrow{\psi}\rr^{n},f\left(c\right)\in V,\psi\left(f\left(c\right)\right)=0$
such that $f\left(U\right)\subset V$ and $\psi\circ f\circ\varphi^{-1}$
is the standard embedding of $\rr_{k}^{n}$ to $\rr^{n}$, restricted
to $\varphi\left(U\right)$.
\end{defn}
\begin{lem}
\label{lem:uniqueness for manc blowup}(a) Let $C$ be a set, $M$
a manifold without boundary, and $f:C\to M$ a map of sets. A structure
of a manifold with corners\footnote{That is, a Hausdorff second countable topology on $C$ together with
a suitable maximal atlas.} on $C$ making $f$ a rectilinear map is unique if it exists. 

(b) If there's an open cover $M=\bigcup V_{i}$ such that $f^{-1}\left(V_{i}\right)$
admits a manifold with corners structure making $f|_{f^{-1}\left(V_{i}\right)}$
rectilinear, then $C$ admits a structure of a manifold with corners
making $f$ rectilinear.

(c) For $i=1,2$, let $C_{i}$ be a manifold with corners, $M_{i}$
a manifold without boundary, and $C_{i}\xrightarrow{f_{i}}M_{i}$
a rectilinear map. Let $M_{1}\xrightarrow{g}M_{2}$ be a smooth map.
An interior map $C_{1}\xrightarrow{\tilde{g}}C_{2}$ making the square
\[
\xymatrix{C_{1}\ar[r]\ar[d] & C_{2}\ar[d]\\
M_{1}\ar[r] & M_{2}
}
\]
commute is unique if it exists.
\end{lem}
\begin{proof}
We prove (a). Suppose $C_{1},C_{2}$ are two manifolds with corners
with the same underlying set $C$, making $C\xrightarrow{f}M$ rectilinear.
It suffices to show that the identity map $C_{1}\to C_{2}$ is weakly
smooth. Consider any $c\in C$. For $i=1,2$ let $\left(c\in U_{i}\xrightarrow{\varphi_{i}}\rr_{k_{i}}^{n},f\left(c\right)\in V_{i}\xrightarrow{\psi_{i}}\rr^{n}\right)$
be a pair of coordinate charts satisfying the conditions of Definition
\ref{def:rectilinear map} with respect to the manifolds with corners
structure $C_{i}$ on $C$. $U_{1}\cap U_{2}=f^{-1}\left(V_{1}\cap V_{2}\right)$
must be open in both topologies on $C$. $\psi_{1}\left(V_{1}\cap V_{2}\right)\xrightarrow{\psi_{2}\psi_{1}^{-1}}\psi_{2}\left(V_{1}\cap V_{2}\right)$
is a smooth map between open neighborhoods of $0\in\rr^{n}$, so its
restriction $\varphi_{1}\left(U_{1}\cap U_{2}\right)\to\varphi_{2}\left(U_{1}\cap U_{2}\right)$
is weakly smooth map.

We prove (b). For each pair of indices $f^{-1}\left(V_{i}\cap V_{j}\right)$
is an open subset of both $f^{-1}\left(V_{i}\right)$ and of $f^{-1}\left(V_{j}\right)$
(since $V_{i}\cap V_{j}$ is an open subset of $V_{i}$ and of $V_{j}$)
and inherits a manifold with corners structure from both. By part
(a), these structure must be equal. In particular the maps $f^{-1}\left(V_{i}\cap V_{j}\right)\to f^{-1}\left(V_{i}\right)$
are continuous (in fact, open topological embeddings) and we can equip
$C$ with the colimit topology, which is clearly second countable.
It remains only to check that this topology is Hausdorff: consider
two distinct points $x,y\in C$. If $f\left(x\right)\neq f\left(y\right)$
there are open neighborhoods $U_{x},U_{y}\subset M$ of $x$ and $y$
respectively with $U_{x}\cap U_{y}=\emptyset$, and we may assume
without loss of generality that $U_{x},U_{y}$ are contained in some
$V_{i},V_{j}$ so that their inverse images are open. If $f\left(x\right)=f\left(y\right)\in V_{i}$
then we can use the assumption that $f^{-1}\left(V_{i}\right)$ is
Hausdorff to separate $x$ from $y$.

We prove (c). We can consider the interior points $\mathring{C}_{i}$
as subsets of $M_{i}$. If $\tilde{g}$ is an interior map then $\tilde{g}\left(\mathring{C}_{1}\right)\subset\mathring{C}_{2}$
so the restriction of $\tilde{g}$ to interior points is determined
by $g$, $\tilde{g}|_{\mathring{C}_{1}}=g|_{\mathring{C}_{1}}$. Since
the interior points are dense in $C_{1}$, a continuous extension
of $g|_{\mathring{C}_{1}}$ to $C_{1}$ is unique if it exists.
\end{proof}
\begin{proof}
[Proof (of Proposition \ref{prop:hyper bu functor ManC})]Let $h_{0}:[k]\times\rr^{n-1}\to\rr^{n}$
denote the standard immersion of the first $k$ coordinate hyperplanes,
and write $E_{0}=\im h_{0}$. It is a hyper map and there's an obvious
identification $B\left(\rr^{n},E_{0}\right)=2^{k}\times\rr_{k}^{n}$
(a choice of germ of connected component of $\rr^{n}\backslash E_{0}$
amounts to choosing a point of $\rr^{n}$ together with an incident
orthant). Clearly, the manifold with corners structure on $2^{k}\times\rr_{k}^{n}$
makes the map $B\left(\rr^{n},E_{0}\right)\to\rr^{n}$ rectilinear.
Now suppose $\left(X,E\right)$ is an object of $\text{\textbf{Man}}^{+}$,
and let $W\xrightarrow{h}X$ be a hyper map with $E=\im h$. We can
cover $X$ by orthant charts, $X=\bigcup V_{i}$ with $V_{i}\xrightarrow{\psi_{i}}\rr^{n}$
as in Proposition \ref{prop:orthant charts exist} so $B\left(V_{i},E\cap V_{i}\right)\subset B\left(X,E\right)$
inherits the structure of a manifold with corners from $B\left(V_{i},E\right)\subset B\left(\rr^{n},E_{0}\right)=2^{k}\times\rr_{k}^{n}$
making the map $B\left(V_{i},E\right)\to V_{i}$ rectilinear. It follows
that $B\left(X,E\right)$ admits a unique manifold with corners structure
making the map to $X$ rectilinear, and this defines the action of
the functor $B$ on objects. Its action on multi-transverse arrows
is defined similarly, using part (b) of Proposition \ref{prop:orthant charts exist}
and the uniqueness of lifts, part (c) of Lemma \ref{lem:uniqueness for manc blowup}.
This uniqueness also implies that the functor is faithful.

Clearly, a local diffeomorphism is multi-transverse to any $\left(X,E\right)$
and its lift admits local inverses, proving the last statement.
\end{proof}

\subsubsection{\label{subsec:orbifold hyper bu}Hyperplane blowup of orbifolds}

We consider the bicategory $\text{\textbf{PEG}}_{\partial=\emptyset}^{+}$
of \emph{marked proper étale groupoids} (without boundary). The objects
of $\text{\textbf{PEG}}_{\partial=\emptyset}^{+}$ are pairs $\left(X_{1}\rightrightarrows X_{0},E\right)$
where $X_{1}\rightrightarrows X_{0}$ is a proper étale groupoid without
boundary, and $E\subset X_{0}$ is a hyper subset \emph{which is a
union of orbits}, $s^{-1}E=t^{-1}E$.

If $\left(X_{\bullet}^{\left(1\right)},E^{\left(1\right)}\right),\left(X_{\bullet}^{\left(2\right)},E^{\left(2\right)}\right)$
are two objects, a 1-cell of $\text{\textbf{PEG}}_{\partial=\emptyset}^{+}$
consists of a smooth functor 
\[
X_{\bullet}^{\left(1\right)}\xrightarrow{F=\left(F_{0},F_{1}\right)}X_{\bullet}^{\left(2\right)}
\]
such that $F_{0}$ is multi-transverse to $E^{\left(2\right)}$ and
$E^{\left(1\right)}=F_{0}^{-1}E^{\left(2\right)}$. Note that every
étale map, and in particular every refinement, satisfies this condition.
The 2-cells in $\text{\textbf{PEG}}_{\partial=\emptyset}^{+}$ are
all the 2-cells of $\text{\textbf{PEG}}$ spanned by the 1-cells specified
above. 

For emphasis, in this subsection we denote the bicategory of proper
étale groupoids and orbifolds in the category $\manc$ by $\text{\textbf{PEG}}^{c}$
and $\text{\textbf{Orb}}^{c}$, respectively. We denote by $\text{\textbf{PEG}}_{ps}^{c},\text{\textbf{Orb}}_{ps}^{c}$
the subcategories whose maps are perfectly simple maps. 

If $X_{1}\rightrightarrows X_{0}$ is a groupoid, we write 
\[
X_{2}=X_{1}\fibp{t}{s}X_{1}
\]
for the manifold with corners parameterizing composable arrows
\[
x_{1}\xrightarrow{a}x_{2}\xrightarrow{b}x_{3}
\]
and, for $i=1,2,3$, we denote by $p_{i}:X_{2}\to X_{0}$ the map
sending a composable arrow as above to $x_{i}$.
\begin{thm}
The functor $B$ extends to a strict 2-functor
\[
B:\text{\textbf{PEG}}_{\partial=\emptyset}^{+}\to\text{\textbf{PEG}}_{ps}^{c}
\]
which takes
\[
\left(\xx=X_{2}\xrightarrow{m}\overset{\overset{i}{\curvearrowleft}}{X_{1}}\begin{array}{c}
\overset{s,t}{\rightrightarrows}\\
\underset{e}{\leftarrow}
\end{array}X_{0},E\right)
\]
to 
\[
B\left(\xx\right)=B\left(X_{2},p_{1}^{-1}E\right)\overset{B\left(m\right)}{\to}B\overset{\overset{B\left(i\right)}{\curvearrowleft}}{\left(X_{1},s^{-1}E\right)}\begin{array}{c}
\overset{B\left(s\right),B\left(t\right)}{\rightrightarrows}\\
\underset{B\left(e\right)}{\leftarrow}
\end{array}B\left(X_{0},E\right)
\]
together with the obvious strict natural transformation $B\left(\xx\right)\xrightarrow{\beta_{\left(\xx,E\right)}}\xx$.

This functor takes refinements to refinements, and thus there's an
induced functor between the 2-localization of these categories 
\[
B:\text{\textbf{Orb}}_{\partial=\emptyset}^{+}\to\text{\textbf{Orb}}_{ps}^{c}.
\]
\end{thm}

\begin{proof}
The verification that $B\left(\xx\right)$ is a groupoid with étale
structure maps is straightforward. We write $B_{0}=B\left(X_{0},E\right)$,
$B_{1}=B\left(X_{1},s^{-1}E\right)$, $B_{2}=B\left(X_{2},p_{1}^{-1}E\right)$
and, for $i=0,1,2$, $B_{i}\xrightarrow{\beta_{i}}X_{i}$ for the
associated map. We need to check that $B\left(s\right)\times B\left(t\right):B_{1}\to B_{0}\times B_{0}$
is proper. Note $\beta_{1}$ is proper, since it is closed (as any
continuous map from a locally compact space to a Hausdorff space is
closed) and the fiber over a point is compact (in fact, finite). It
follows that $\left(s\times t\right)\circ\beta_{1}$ is proper and
so, if $K\subset B_{0}\times B_{0}$ is any compact subset, 
\[
\left(B\left(s\right)\times B\left(t\right)\right)^{-1}\left(K\right)\subset\left(\left(s\times t\right)\circ\beta_{1}\right)^{-1}\left(\left(\beta_{0}\times\beta_{0}\right)\left(K\right)\right)
\]
is the inclusion of a closed subset into a compact subset; therefore
$\left(B\left(s\right)\times B\left(t\right)\right)^{-1}\left(K\right)$
must be compact.

The other statements are also straightforward.
\end{proof}

\subsubsection{A hyper map between orbifolds}

There's a natural way to construct objects and arrows in $\text{\textbf{Orb}}_{\partial=\emptyset}^{+}$.
Let $h:\mathcal{W}\to\xx$ be a map of orbifolds without boundary,
given by a pair of smooth functors
\begin{equation}
\wc\xleftarrow{S}\left(\tilde{\wc}=\tilde{W}_{1}\rightrightarrows\tilde{W}_{0}\right)\xrightarrow{H=\left(H_{1},H_{0}\right)}\left(\xx=X_{1}\rightrightarrows X_{0}\right)\label{eq:h as a fraction}
\end{equation}
with $S$ a refinement. 

Let $W_{0}'=X_{1}\fibp{t}{H_{0}}\tilde{W}_{0}$ . Since $t$ is étale
this fiber product exists. We let $H_{0}':W_{0}'\to X_{0}$ denote
the composition $X_{1}\fibp{t}{H_{0}}\tilde{W}_{0}\to X_{1}\xrightarrow{s}X_{0}$.
We call the image of $H_{0}'$ the\emph{ essential image} of $h$,
and denote it $\im h$. Fix some point $x\in X_{0}$. The \emph{essential
fiber }of $h$ over $x$ is a topological groupoid, with object space
$\left(H_{0}'\right)^{-1}\left(x\right)$ and with arrows between
$\left(x\xrightarrow{\alpha}H_{0}\left(w\right),w\right)$ and $\left(x\xrightarrow{\alpha'}H_{0}\left(w'\right),w'\right)$
consisting of the arrows in $\tilde{W}_{1}$ between $w$ and $w'$
(this is a special case of the weak fiber product, see Lemma \ref{lem:fibp in orb}).
If $R$ is a refinement, the essential fiber of $h$ over $x$ and
of $R\circ h$ over $R\left(x\right)$ are equivalent. The essential
image and, up to equivalence the essential fiber, depend only on
the homotopy class of $H$ (in particular, they do not depend on $S$).
\begin{defn}
\label{def:orbimaps properies 2}We say that $h$ is \emph{hyper}
if the following five conditions are met (cf. Definition \ref{def:hyper})
\begin{itemize}
\item $h$ is \emph{faithful}, which means $H_{0}$ is faithful. This implies
the essential fiber over every point is equivalent to a set (with
a topology).
\item $h$ is a \emph{closed immersion}, which means $H_{0}$ is a closed
immersion. This implies the orbit space of each essential fiber has
the discrete topology.
\item $h$ is \emph{proper}, which means the essential fibers have compact
orbit spaces. Given our previous assumptions this means the\emph{
}essential fiber is equivalent to a \emph{finite} set\emph{,} and
and we fix representatives $\left\{ q_{i}=x\xrightarrow{\alpha_{1}}H_{0}\left(w_{1}\right)\right\} _{i=1}^{r}$.
\item We require that $h$ has \emph{transversal self-intersection},\emph{
}that is, we require the map
\[
\bigoplus_{i=1}^{r}N_{w_{i}}^{\vee}\to T_{x}^{\vee}X_{0}
\]
be injective, where $N_{w_{i}}^{\vee}=\ker\left(T_{w_{i}}^{\vee}W_{0}\to T_{x}^{\vee}X_{0}\right)$;
this is independent of the choice of representatives.
\item $h$ has \emph{codimension one}, meaning $\dim\xx-\dim\wc=1$.
\end{itemize}
If $\yy$ is another orbifold without boundary, we say a map $\yy\xrightarrow{f}\xx$
is \emph{multi-transverse} to $h$ if $h$ is (b-)transverse to $f$
and the 2-pullback $f^{-1}h$ has transversal self-intersection (it
is automatically a proper, faithful closed immersion). 
\end{defn}
\begin{lem}
(a) Let $\wc\xrightarrow{h}\xx$ be a hyper map. Then $\left(\xx,\im h\right)$
is an object of $\text{\textbf{Orb}}_{\partial=\emptyset}^{+}$.

(b) If $f$ is multi-transverse to $h$, $\left(\yy,\im f^{-1}h\right)\xrightarrow{f}\left(\xx,\im h\right)$
is an arrow in $\text{\textbf{Orb}}_{\partial=\emptyset}^{+}$.
\end{lem}
\begin{proof}
Straightforward.
\end{proof}

\subsection{\label{subsec:Group-actions}Group actions}

Let $G$ be a compact lie group, with multiplication ${m:G\times G\to G}$
and identity $e:\pt\to G$. Given a bicategory of spaces $\cl$ such
as\footnote{More precisely, we need to be able to consider $G$ as a group object
in $\cl$ and certain products to exist. For $\cl=\text{\textbf{Orb}}_{\partial=\emptyset}^{+}$
we consider $G$ as having the trivial marking $\emptyset$, and we
have 
\[
\left(\xx,E\right)\times\left(\yy,F\right)=\left(\xx\times\yy,E\times\yy\coprod\xx\times F\right).
\]
} $\manc,\text{\textbf{Orb}},\text{\textbf{Orb}}_{\partial=\emptyset}^{+}$,
we construct a category $G-\cl$ of $G$-equivariant objects following
Romagny \cite{stacky-action}. We briefly explain how to translate
his definitions to our setup, and refer the reader to \cite{stacky-action}
for more details. A 0-cell of $G-\cl$ is given by a 4-tuple $\left(\xx,\mu,\alpha,a\right)$
where $\xx$ is a 0-cell of $\cl$, $\mu:G\times\xx\to\xx$ is a 1-cell,
and $\alpha$ and $a$ are 2-cells filling in, respectively, the following
square and triangle:
\[
\xymatrix{G\times G\times\xx\aru{r}{m\times\id_{\xx}}\ar[d]_{\id_{G}\times\mu} & G\times\xx\ar[d]^{\mu}\\
G\times\xx\ar[r]_{\mu} & \xx
}
\;\xymatrix{G\times\xx\ar[r]^{\mu} & \xx\\
\xx\ar[ur]_{\id_{\xx}}\ar[u]^{e\times\id_{\xx}}
}
.
\]
A 1-cell (or \emph{$G$-equivariant map})
\[
\left(\xx,\mu,\alpha,a\right)\to\left(\xx',\mu',\alpha',a'\right)
\]
is given by a pair $\left(\xx\xrightarrow{f}\xx',\sigma\right)$ where
$\sigma$ is a 2-cell filling in the square 
\[
\xymatrix{G\times\xx\ar[r]^{\mu}\ar[d]_{\id_{G}\times f} & \xx\ar[d]^{f}\\
G\times\xx'\ar[r]_{\mu'} & \xx'
}
.
\]
A 2-cell $\left(f,\sigma\right)\Rightarrow\left(f',\sigma'\right)$
is given by a 2-cell $f\xRightarrow{\beta}f'$. As usual, the 2-cells
$\alpha,a,\sigma,\beta$ are required to satisfy some coherence conditions,
cf. \cite[Definition 2.1]{stacky-action}. 

\subsubsection{The fixed-points of a hyperplane blowup.}

There's an obvious pair of pseudofunctors $\iota:\cl\to G-\cl$ and
$f:G-\cl\to\cl$; $\iota$ equips every 0-cell with the trivial $G$-action
and every 1-cell with the trivial $G$-equivariant structure, $f$
forgets the extra structure. Let $\xx$ be a 0-cell of $G-\cl$. Consider
the pseudofunctor $\xx_{!}^{G}:\cl^{op}\to Cat$ 
\[
\xx_{!}^{G}\left(-\right)=Hom_{G-\cl}\left(\iota\left(-\right),\xx\right),
\]
together with the obvious pseudonatural transformation $\xx_{!}^{G}\to Hom_{\cl}\left(-,f\left(\xx\right)\right)$.
$\xx_{!}^{G}$ may or may not be represented by an object of $\cl$;
If it is, we denote the representing object by $\xx^{G}$ and say
``the fixed-point locus $\xx^{G}$ exists (as an object of $\cl$)''
(cf. \cite[Definition 2.3]{stacky-action}). Here are a few examples.
\begin{itemize}
\item $\cl$ is the category of stacks: the fixed-point locus always exist,
see \cite[Proposition 2.5]{stacky-action} (though for a general stack,
this may be quite ill-behaved). 
\item $\cl=\text{\textbf{Orb}}$: it follows from the slice theorem that
the fixed point locus $\xx^{G}$ of any $G$-orbifold without boundary
$\xx$ exists, and the map $\xx^{G}\to\xx$ is a closed embedding. 
\item $\cl=\manc$ (considered as a bicategory with only identity 2-cells),
$G=\tb$ a torus group: we do not know whether the fixed-point set
exists in general.
\end{itemize}
The last example motivates restricting our attention to manifolds
(and orbifolds) with corners which are obtained as a hyperplane blowup.
We will see that if $\widetilde{\xx}$ is obtained as a hyperplane
blowup of $\xx$ at an invariant hyper subset, then the fixed-points
of $\widetilde{\xx}$ exist and the map $\widetilde{\xx}^{G}\to\widetilde{\xx}$
is a closed embedding. The remainder of this subsection is devoted
to formulating and proving this result.

Any pseudofunctor $\cl_{1}\to\cl_{2}$, where $\cl_{i}$ is a bicategory
of spaces as above for $i=1,2$, induces a functor between the corresponding
$G$-categories. In particular, the blow up $B:\text{\textbf{Orb}}^{+}\to\text{\textbf{Orb}}_{ps}^{c}$
induces a functor
\[
G-\text{\textbf{Orb}}^{+}\xrightarrow{B}G-\text{\textbf{Orb}}_{ps}^{c}.
\]

\begin{prop}
Let $G=U\left(1\right)^{m}$ be a torus group, and let $\widetilde{\xx}=B\left(\xx,E\right)$
denote the blowup of an object of $G$-$\text{\textbf{Orb}}^{+}$
(we're supressing the $G$-action data from the notation). Let $i:\xx^{G}\to\xx$
be the closed embedding of the fixed-points of $\xx$. Then 

(a) the fixed points $\left(\widetilde{\xx}\right)^{G}$ exist as
an object of $\text{\textbf{Orb}}^{c}$, and the map $\left(\widetilde{\xx}\right)^{G}\xrightarrow{\tilde{i}}\widetilde{\xx}$
is a closed embedding.

(b) $i$ is multi-transverse to $h$ and there's a natural isomorphism
\[
\left(\widetilde{\xx}\right)^{G}\simeq\left(\xx^{G}\right)^{\sim}=:B\left(\xx^{G},\im i^{-1}h\right),\text{ and}
\]
(c) there's a 2-cell making the following square cartesian 
\[
\xymatrix{\left(\xx^{G}\right)^{\sim}=\left(\widetilde{\xx}\right)^{G}\ar[r]\ar[d] & \widetilde{\xx}\ar[d]\\
\xx^{G}\ar[r] & \xx
}
\]
where the horizontal maps are the structure maps of the fixed points
and the vertical maps are the maps associated with the hyperplane
blowup.
\end{prop}
\begin{proof}
We describe the argument when $\widetilde{\xx}=B\left(X,E\right)$
is the blowup of a marked manifold with corners, and leave the general
case to the reader. Note that the statement is local, so we fix some
$p\in X^{G}$ and work throughout with germs of functions around $p$
and its images (i.e., we will \emph{not} specify the open subsets
on which each map is defined). Using Proposition \ref{prop:orthant charts exist}
we reduce to the following setup.

Let 
\[
W=\left[k\right]\times\rr^{n-1}\xrightarrow{h}X=\rr^{n}
\]
denote the standard hyper map, so for $1\leq i\leq k$, $\left\{ i\right\} \times\rr^{n-1}\xrightarrow{h_{i}}\rr^{n}$
is the map 
\begin{equation}
\left(t_{1},...,t_{n-1}\right)\mapsto\left(t_{1},...,t_{i-1},0,t_{i},...,t_{n-1}\right).\label{eq:std h_i form-1}
\end{equation}
we assume that $G=\tb$ acts compatibly on the domain and range of
$h$; since $G$ is connected this means each $h_{i}$ is $G$-equivariant.
We take $p$ to be the origin $0\in\rr^{n}$, and write $\left\{ q_{i}=\left\{ i\right\} \times0\right\} _{i=1}^{k}$
for its preimages under $h$. 

Although a-priori, the action in these coordinates is \emph{not }linear,
it does preserve the zeros of the first $k$ components, and thus
also their signs. More precisely, if we define the signature $\sigma\left(x\right)\in\left\{ -,0,+\right\} ^{k}$
of a point $\left(x_{1},...,x_{n}\right)\in\rr^{n}$ by recording
the sign of the first $k$ coordinates, then $\sigma\left(g.x\right)=\sigma\left(x\right)$
for all $g,x$. We define a new orthant chart $\left(x_{1}'\left(q\right),...,x_{n}'\left(q\right)\right)$
(in a perhaps smaller neighborhood of the origin) by
\[
x_{i}'\left(q\right)=\begin{cases}
\int_{g\in G}x_{i}\left(g.q\right)\,dH\left(g\right) & 1\leq i\leq k\\
x_{i}\left(q\right) & k+1\leq i\leq n
\end{cases}
\]
where $dH$ denotes the Haar measure on $G$. Clearly, $x_{i}'$ are
$G$-invariant and $x_{i}'\circ h_{i}\equiv0$ for $1\leq i\leq k$;
the key point is that, since $G$ admits no non-trivial rank one representations,
the lines $\ker\left(T_{p}^{*}X\to T_{q_{i}}^{*}W\right)$ are fixed,
and thus $dx_{i}'|_{0}=dx_{i}|_{0}$ for $1\leq i\leq n$; in particular
$\left(dx'_{i}|_{0}\right)_{i=1}^{n}$ form a basis for $T_{p}^{*}X$. 

We can choose a basis $v_{1}^{*},...,v_{n}^{*}$ to $T_{p}^{*}X$
so that $v_{i}^{*}=dx_{i}'|_{0}$ for $1\leq i\leq k$ and 
\[
v_{k+1}^{*},...,v_{s}^{*},v_{s+1}^{*}+\sqrt{-1}v_{s+2}^{*},...,v_{n-1}^{*}+\sqrt{-1}v_{n}^{*}
\]
are common eigenvectors for the infinitesimal generators $\xi_{1},...,\xi_{m}\in End\left(T_{p}^{*}X\right)$
of the linearized action. Fix a $G$-invariant metric $R$, let $v_{1},...,v_{n}$
be a basis for $T_{p}X$ dual to $v_{1}^{*},...,v_{n}^{*}$, and consider
exponential coordinates $\left(y_{1},...,y_{n}\right)$ for $X=\rr^{n}$
around $p=0$ so that 
\[
\left(y_{1},...,y_{n}\right)\mapsto\exp_{R}\left(\sum y_{i}v_{i}\right).
\]

Now consider the coordinates

\[
\left(x_{1}',...,x_{k}',y_{k+1},...,y_{n}\right).
\]
These define an orthant chart. Moreover, in these coordinates the
action is the linear action on
\[
\rr^{n}=\rr^{s}\oplus\bigoplus_{i=1}^{r}\cc_{\lambda_{i}}
\]
fixing $\rr^{s}$, $s\geq k$ and acting on $\cc_{\lambda_{i}}$ by
a character $\lambda_{i}:\tb\to U\left(1\right)$.

Working with such orthant charts, the verification of properties
(a)-(c) is straightforward.
\end{proof}

\bibliographystyle{amsabbrvc}
\bibliography{localization}

\providecommand{\bysame}{\leavevmode\hbox to3em{\hrulefill}\thinspace}
\providecommand{\MR}{\relax\ifhmode\unskip\space\fi MR }
\providecommand{\MRhref}[2]{%
  \href{http://www.ams.org/mathscinet-getitem?mr=#1}{#2}
}
\providecommand{\href}[2]{#2}
\begin{thebibliography}{10}

\bibitem{behrend}
K.~Behrend, \emph{Cohomology of stacks}, Intersection theory and moduli, ICTP
  Lect. Notes, XIX, Abdus Salam Int. Cent. Theoret. Phys., Trieste, 2004,
  pp.~249--294 (electronic).

\bibitem{benabou}
J.~B{\'e}nabou, \emph{Introduction to bicategories}, Reports of the {M}idwest
  {C}ategory {S}eminar, Springer, Berlin, 1967, pp.~1--77.

\bibitem{fulton-pandharipande}
W.~Fulton and R.~Pandharipande, \emph{Notes on stable maps and quantum
  cohomology}, 1996, \href {http://arxiv.org/abs/arXiv:alg-geom/9608011}
  {\path{arXiv:arXiv:alg-geom/9608011}}.

\bibitem{joyce-generalized}
D.~Joyce, \emph{A generalization of manifolds with corners}, \href
  {http://arxiv.org/abs/http://arxiv.org/abs/1501.00401v2}
  {\path{arXiv:http://arxiv.org/abs/1501.00401v2}}.

\bibitem{joyce-fibered}
\bysame, \emph{On manifolds with corners}, \href
  {http://arxiv.org/abs/http://arxiv.org/abs/0910.3518v2}
  {\path{arXiv:http://arxiv.org/abs/0910.3518v2}}.

\bibitem{generalized2ordinary}
C.~Kottke and R.~B. Melrose, \emph{Generalized blow-up of corners and fiber
  products},  (2011), \href {http://arxiv.org/abs/arXiv:1107.3320}
  {\path{arXiv:arXiv:1107.3320}}, \href
  {http://dx.doi.org/10.1090/S0002-9947-2014-06222-3}
  {\path{doi:10.1090/S0002-9947-2014-06222-3}}.

\bibitem{lerman}
E.~Lerman, \emph{Orbifolds as stacks?}, \href
  {http://arxiv.org/abs/http://arxiv.org/abs/0806.4160v2}
  {\path{arXiv:http://arxiv.org/abs/0806.4160v2}}.

\bibitem{LiuModuli}
C.-C.~M. Liu, \emph{Moduli of j-holomorphic curves with lagrangian boundary
  conditions and open gromov-witten invariants for an $s^1$-equivariant pair},
  2002, \href {http://arxiv.org/abs/arXiv:math/0210257}
  {\path{arXiv:arXiv:math/0210257}}.

\bibitem{mcduff+salamon}
D.~McDuff and D.~Salamon, \emph{{$J$}-holomorphic curves and symplectic
  topology}, second ed., American Mathematical Society Colloquium Publications,
  vol.~52, American Mathematical Society, Providence, RI, 2012.

\bibitem{metzler}
D.~Metzler, \emph{Topological and smooth stacks}, 2003, \href
  {http://arxiv.org/abs/arXiv:math/0306176} {\path{arXiv:arXiv:math/0306176}}.

\bibitem{moerdijk-classifying-groupoids}
I.~Moerdijk, \emph{The classifying topos of a continuous groupoid. {I}}, Trans.
  Amer. Math. Soc. \textbf{310} (1988), no.~2, 629--668.

\bibitem{pronk}
D.~A. Pronk, \emph{Etendues and stacks as bicategories of fractions},
  Compositio Math. \textbf{102} (1996), no.~3, 243--303.

\bibitem{moduli-maps}
J.~W. Robbin, Y.~Ruan, and D.~A. Salamon, \emph{The moduli space of regular
  stable maps}, Math. Z. \textbf{259} (2008), no.~3, 525--574, \href
  {http://dx.doi.org/10.1007/s00209-007-0237-x}
  {\path{doi:10.1007/s00209-007-0237-x}}.

\bibitem{stacky-action}
M.~Romagny, \emph{Group actions on stacks and applications}, Michigan Math. J.
  \textbf{53} (2005), no.~1, 209--236.

\bibitem{jake+sara-A8}
J.~P. Solomon and S.~B. Tukachinsky, \emph{Differential forms, fukaya
  $a_\infty$ algebras, and {Gromov-Witten} axioms}, 2016, \href
  {http://arxiv.org/abs/arXiv:1608.01304} {\path{arXiv:arXiv:1608.01304}}.

\bibitem{wfibp}
M.~Tommasini, \emph{Weak fiber products in a bicategory of fractions}, 2014,
  \href {http://arxiv.org/abs/arXiv:1412.3295} {\path{arXiv:arXiv:1412.3295}}.

\bibitem{fp-loc-OGW}
A.~N. Zernik, \emph{Fixed-point localization for open {Gromov-Witten} theory}.

\bibitem{twA8}
\bysame, \emph{Equivariant {A-infinity} algebras for nonorientable
  lagrangians}, 2015, \href {http://arxiv.org/abs/arXiv:1512.04507}
  {\path{arXiv:arXiv:1512.04507}}.

\bibitem{equiv-OGW-invts}
\bysame, \emph{Equivariant open {Gromov-Witten} theory of
  $(\mathbb{C}\mathbb{P}^{2m}, \mathbb{R}\mathbb{P}^{2m})$}, preprint.

\end{thebibliography}

\end{document}